\documentclass[letterpaper,11pt]{amsart}
\usepackage{amsmath,amsthm,amssymb,mathtools,amsfonts,mathrsfs}
\usepackage{xspace,xcolor}
\usepackage[breaklinks,colorlinks,citecolor=teal,linkcolor=teal,urlcolor=teal,pagebackref,hyperindex]{hyperref}
\usepackage[alphabetic]{amsrefs}
\usepackage{tikz-cd}
\usepackage[all]{xy}
\usepackage{bbm}
\usepackage{tikz-cd} 
\usepackage{MnSymbol}
\usepackage{stmaryrd}

\usepackage[bbgreekl]{mathbbol}

\DeclareSymbolFontAlphabet{\mathbb}{AMSb}
\DeclareSymbolFontAlphabet{\mathbbl}{bbold}

\newtheorem{theorem}{Theorem}[section]

\newtheorem{proposition}[theorem]{Proposition}
\newtheorem{corollary}[theorem]{Corollary}

\newtheorem{lemma}[theorem]{Lemma}

\theoremstyle{definition} 
\newtheorem{defn}[theorem]{Definition}

\newtheorem{remark}[theorem]{Remark}

\newtheorem{question}[theorem]{Question}

\newcommand{\cO}{\mathcal{O}}

\newcommand{\lie}[1]{\ensuremath{\mathfrak{#1}}}
\newcommand{\g}{{\lie{g}}}

\newcommand{\R}{\bb{R}}

\newcommand{\qu}{/\kern-.7ex/}
\newcommand{\lqu}{\backslash \kern-.7ex \backslash}
\newcommand{\on}{\operatorname} 
\newcommand{\Aut}{\on{Aut}}
 
\newcommand{\Hom}{\on{Hom}}

\title{Gromov--Witten theory beyond maximal contacts}

\author{Yu Wang}
\address{Department of Mathematics \\ Louisiana State University \\303 Lockett Hall \\ Baton Rouge \\ LA \\ United States of America \\ 70803}
\email{yuwang@lsu.edu}

\author{Fenglong You}
\address{Department of Mathematics \\ ETH Z\"urich, \\Rämistrasse 101, \\8092 Zürich, \\Switzerland}
\email{fenglong.you@math.ethz.ch}

\address{School of Mathmatical Sciences \\ University of Nottingham, \\University Park \\Nottingham, NG7 2RD \\United Kingdom.}
\email{fenglong.you@nottingham.ac.uk}

\thanks{}


\begin{document}
\date{\today}

\begin{abstract} 
Given a smooth projective variety $X$ and a smooth nef divisor $D$, we identify genus zero relative Gromov--Witten invariants of $(X,D)$ with $(n+1)$ relative markings with genus zero orbifold Gromov--Witten invariants of multi-root stacks over the $\mathbb P^1$-bundle $P:=\mathbb P(\mathcal O_X(-D)\oplus \mathcal O_X)$ with $n$ orbifold markings. This is a generalization of the local-relative correspondence beyond maximal contacts. Repeating this process, we identify genus zero relative Gromov--Witten invariants of ambient insertions with genus zero absolute Gromov--Witten invariants of toric bundles. 
We also present how this correspondence can be used to compute genus zero two-point relative Gromov--Witten invariants. 
\end{abstract}

\maketitle 

\tableofcontents

\section{Introduction}

\subsection{The degeneration scheme}

Relative Gromov--Witten theory was introduced in \cite{LR}, \cite{IP}, \cite{Li1}, \cite{Li2} in order to apply the degeneration technique to study Gromov--Witten theory. When the central fiber of the degeneration has two irreducible components glued along a smooth divisor, the degeneration formula relates absolute Gromov--Witten invariants with relative Gromov--Witten invariants of the degenerated pieces. Relative Gromov--Witten invariants that appear in the degeneration formula are usually with several relative markings with possibly primitive insertions. There have been tremendous advances in understanding Gromov--Witten theory via degenerations. For example, the degeneration scheme proposed by \cite{MP}. However, as far as we know, computing invariants with several relative markings in general, even in genus zero, is still very difficult.

The simplest type of relative invariants are genus zero relative invariants with one relative marking. These invariants can be computed via the relative $I$-functions in the relative mirror theorem \cite{FTY}*{Theorem 1.4} when the divisor is nef. 

For many genus zero relative invariants with several relative markings, in principle, one can also compute them via the extended relative $I$-functions in \cite{FTY}*{Theorem 1.5}. An example of such computation is presented in \cite{You22}, where the author computed the proper Landau--Ginzburg potentials, which are generating functions of two-point relative Gromov--Witten invariants of smooth log Calabi--Yau pairs. It can be seen in \cite{You22} that the computation is much more complicated than the computation of one-point relative Gromov--Witten invariants. The computation in \cite{You22} is not just applying the relative mirror theorem in \cite{FTY}*{Theorem 1.5} for extended $I$-functions, one also needs to understand the extended relative mirror maps and prove several non-trivial identities for relative Gromov--Witten invariants. The computation for more general two-point relative invariants or relative invariants with more than two relative markings is even more complicated. 

There can not be a general and effective degeneration scheme without understanding relative Gromov--Witten invariants with more than one relative marking. The goal of this article is to determine genus zero relative Gromov--Witten invariants with several relative markings effectively.

The following invariants are related to relative invariants and can be used to study relative invariants. Here we focus on the genus zero case:
\begin{enumerate}
    \item Local Gromov--Witten invariants when the divisor is nef: By \cite{vGGR}, they are relative invariants with maximal contact and ambient insertions.
    \item Orbifold Gromov--Witten invariants of root stacks: By \cite{ACW} and \cite{FWY}, they coincide with relative invariants when the root $r$ is taken to be sufficiently large. When the divisor is snc, the large $\vec r$-limits of orbifold Gromov--Witten invariants of multi-root stacks are defined in \cite{TY20c}. When the irreducible components of the divisor are nef, these orbifold invariants (with maximal contacts) coincide with local invariants (of higher ranks) by \cite{BNTY}.
    \item Logarithmic Gromov--Witten invariants defined by \cite{Kim10}, \cite{GS13}, \cite{Chen}, \cite{AC} and \cite{ACGS}: By \cite{AMW} they coincide with relative invariants when the divisor is smooth. But different from local and orbifold invariants by \cite{NR}, \cite{BNR22}, \cite{BNR24}.
    \item Naive logarithmic Gromov--Witten invariants for snc pairs: by \cite{BNR22} and \cite{BNR24}, they coincide with orbifold invariants.
\end{enumerate}
We give more detailed explanation of (1) and (2), which are invariants that we will study in this paper.

\subsection{The local-relative correspondence}

One way to study genus zero relative Gromov--Witten invariants is through the local-relative correspondence.

Let $X$ be a smooth projective variety and let $D\subset X$ be a smooth divisor. When $D$ is nef, genus zero relative Gromov--Witten invariants of $(X,D)$ with maximal contact orders are genus zero relative invariants with only one relative marking. It was first conjectured by Takahashi \cite{Takahashi} that genus zero maximal contact invariants of $\mathbb P^2$ relative to a smooth cubic $E$ are equal to local Gromov--Witten invariants of the total space of $K_{\mathbb P^2}$ (up to a constant factor and a sign). The $(\mathbb P^2,E)$-case was proved by Gathmann by explicit computation on both sides \cite{Gathmann03}. In general, maximal contact invariants are local invariants of the total space of the bundle $\mathcal O_{X}(-D)$ by the work of \cite{vGGR}. Some generalizations of \cite{vGGR} are proved in \cite{TY20b}, \cite{BFGW} and \cite{BNTY}. The relation between relative invariants with maximal contacts and local invariants is referred to as the local-relative correspondence or the local-log principle. Since, for snc pairs, local invariants and log invariants are different  in general \cite{NR}, we decided to use the term ``the local-relative correspondence'' instead of ``the local-log principle'' to avoid any confusion. 

Another way to study the relative Gromov--Witten theory is to consider it as a limit of the orbifold Gromov--Witten theory of root stacks $X_{D,r}$ for $r$ sufficiently large \cite{ACW}, \cite{TY18}, \cite{FWY}, \cite{TY20b}, \cite{TY20c}. We refer to this as the relative-orbifold correspondence. Through this correspondence, one-point invariants (with maximal contact) can also be computed via the relative mirror theorem \cite{FTY} when the divisor $D$ is nef. 

Compared with the local-relative correspondence, the relative-orbifold correspondence is not restricted to the maximal contact invariants, do not assume the nefness of the divisor $D$ and gives nice structure properties \cite{FWY}. On the other hand, the advantage of the local-relative correspondence is that the geometry of the bundle $\mathcal O_X(-D)$ is usually simpler than the orbifold $X_{D,r}$. For example, if $X$ is a toric variety, then $\mathcal O_X(-D)$ is still toric. But $X_{D,r}$ is not toric when $D$ is not a toric invariant divisor of $X$. Instead, $X_{D,r}$ is a hypersurface of a $\mathbb P^1[r]$-bundle over $X$ \cite{FTY}. In addition, the local-relative correspondence does not require ``the limiting process" as in the relative-orbifold correspondence. Despite the restriction to the maximal contact case, the above advantages of the local-relative correspondence have led to many interesting applications, e.g., the holomorphic anomaly equation for $(\mathbb P^2,E)$ in \cite{BFGW}.

Although so far the local-relative correspondence only applies to the maximal contact case, one can consider it as a motivating case of a more general correspondence. We can think of the local-relative correspondence as a way to reduce the number of relative markings from $1$ to $0$ and increase the target dimension by $1$. There are at least two natural questions that one can ask.

\begin{question}
    What is the geometrical reasoning for the local-relative correspondence? 
\end{question}

As far as we know, besides a heuristic argument for $(\mathbb P^2,E)$ at the end of \cite{Takahashi}, the remarkable observation of Takahashi \cite{Takahashi} was based on computation. There was no direct geometric motivation for the conjecture. The local-relative correspondence for general targets was proved by \cite{vGGR} by degeneration to the normal cone. One may think of the degeneration argument in \cite{vGGR} as a geometrical explanation of the local-relative correspondence. However, at least for us, it is not yet clear how local invariants can encode tangency conditions. If there is a way to encode tangency conditions for local theory, then why is it restricted to the maximal contact case? This leads to the second question.

\begin{question}
Is there such a relation beyond the maximal contact invariants? 
\end{question}

In other words, is it possible to reduce the relative Gromov--Witten theory beyond maximal contacts (that is, with several relative markings) to the (absolute) Gromov--Witten theory of certain higher rank toric bundles over $X$ and use this generalized correspondence to study structures of relative Gromov--Witten theory?

Recall that the local-relative correspondence can be stated as follows (we state it on the invariant level and assume that the virtual dimension constraint is satisfied.):

\begin{align}\label{iden-local-rel-1}
\langle [\iota^*\gamma]_{D\cdot \beta}\rangle_{0,1,\beta}^{(X,D)}=(-1)^{D\cdot \beta+1}\langle [\gamma\cup D]\rangle_{0,1,\beta}^{\mathcal O_X(-D)},
\end{align}
where $\iota:D\hookrightarrow X$ is the inclusion and $\gamma\in H^*(X)$. We remark that the original version of the local-relative correspondence was stated for the case when the relative marking of $(X,D)$ has insertion $[1]_{D\cdot\beta}$ and no marking for the local theory, but an additional factor of $D\cdot\beta$ coming from the divisor equation. As mentioned in \cite{TY20b}*{Equation (2)} and \cite{FW}*{Introduction}, it is not difficult to observe that the proof of the local-relative correspondence in \cite{vGGR} can be easily adapted to prove Identity (\ref{iden-local-rel-1}). The local Gromov--Witten invariant of $\mathcal O_{X}(-D)$ is a twisted Gromov--Witten theory of $X$ by the inverse Euler class of $ O_X(-D)$.  Without this twist, the virtual dimensions of two theories will not match. We refer to \cite{CG} for the definition of the twisted Gromov--Witten theory.

However, it does not seem to be a natural way for the local theory to encode the information of contact orders. For maximal contact invariants, there is only one relative marking and the contact order is fixed when the curve class is fixed. In this case, no serious issues arise.

We may think of the local-relative correspondence as a way to turn a relative marking in the relative Gromov--Witten theory into a relative marking of the local Gromov--Witten theory with contact order $0$. When there are several relative markings, we may want to turn some of the relative markings into relative markings with contact order zero (i.e. interior markings.) Then the question becomes: which divisor of $\mathcal O_{X}(-D)$ is the relative divisor? Unfortunately, for the local Gromov--Witten theory of the total space of the bundle $\mathcal O_{X}(-D)$, there is no natural choice of such a divisor. 

Here comes the key observation: the reason that we do not see the divisor is that the divisor is not in $\mathcal O_{X}(-D)$. The divisor is at infinity! 

We consider the following invariant:
\begin{align}\label{inv-P-0}
(-1)^{D\cdot \beta}\langle [\gamma\cup(X_{\infty}- \pi^*D)]_0\rangle_{0,1,\beta+0f}^{(P,X_{\infty}+ X_{\sigma})},
\end{align}
where $$\pi:P=\mathbb P(\mathcal O_X(-D)\oplus \mathcal O_X)\rightarrow X$$ is a $\mathbb P^1$-bundle over $X$; $X_{\infty}$ is the infinity divisor of $P$; $X_\sigma$ is the graph of the section $\sigma$ of $\mathcal O_X(D)$ that defines $D$, and $X_{\infty}\cap X_{\sigma}=D$; $f$ is the class of a fiber of the bundle $P\rightarrow X$; the invariant is the limiting orbifold invariant of the multi-root stack of $(P,X_{\infty}+ X_{\sigma})$ defined in \cite{TY20c}*{Definition 20}. Note that we have no fiber class here. Let $X_0$ be the zero divisor of $P$. Then the basic divisor relation in $H^2(P)$ is $X_\infty-\pi^*D=X_0$.

It can be proved by degeneration that the local invariant in (\ref{iden-local-rel-1}) is also equal to (\ref{inv-P-0}). The proof will be a special case of the proof of our main theorem. The set-up here is more subtle than the observation that we mentioned above, because we do not just consider tangency conditions along the infinity divisor $X_{\infty}$, but we also consider tangency conditions along another divisor $X_\sigma$. Intuitively, when the curves are in $X_\infty$, we will get the relative Gromov--Witten invariants of $(X_\infty,X_\infty\cap X_\sigma)=(X,D)$ instead of the absolute Gromov--Witten invariants of $X_\infty$.  

With this compactification and the choice of divisors, we are now able to encode some tangency conditions. As we will see later, for relative invariants with $2$ relative markings, the relation is the following:
\begin{align*}
&\langle [\iota^*\gamma_0]_{k_0},[\gamma_1]_{k_1}\rangle_{0,2,\beta}^{(X,D)}\\
=&(-1)^{k_0}\langle [\gamma_0\cup(X_\infty-\pi^*D)]_{(0,0)}, [\gamma_1]_{(k_1,k_1)} \rangle_{0,2,\beta+k_1f_1}^{(P,X_{\infty}+X_{\sigma})},
\end{align*}
where $k_0+k_1=D\cdot \beta$, $\gamma_0\in \iota^*H^*(X)$, $\gamma_1\in H^*(D)$, $\iota: D\hookrightarrow X$ is the inclusion.

For relative invariants with $(n+1)$ relative markings, the relation is the following:
\begin{align*}
&\langle [\iota^*\gamma_0]_{k_0},[\gamma_1]_{k_1},\ldots,[\gamma_n]_{k_n}\rangle_{0,n+1,\beta}^{(X,D)}\\
=&(-1)^{k_0}\langle [\gamma_0\cup(X_\infty-\pi^*D)]_{(0,0)}, [\gamma_1]_{(k_1,k_1)},\ldots,[\gamma_n]_{(k_n,k_n)} \rangle_{0,n+1,\beta+(\sum_{i=1}^n k_i)f}^{(P,X_{\infty}+X_{\sigma})},
\end{align*}
where $k_0+k_1+\cdots+k_n=D\cdot \beta$, $\gamma_0\in H^*(X)$ and $\gamma_i\in H^*(D)$ for $i=1,\ldots,n$.

\if{
Let $P_1:=P$ and $X_{\infty,1}=X_\infty$. If $\gamma_i\in \iota^*H^*(X)$, for $i=1,\ldots,n$, then we can repeat this process by considering $P_2=\mathbb P(\mathcal O_{P_1}(-X_{\infty,1})\oplus \mathcal O_{P_1})$. Let $X_{\infty,2}$ be the infinity divisor of the $\mathbb P^1$-bundle $\pi_2:P_2\rightarrow P_1$. Let $X_{\sigma,2}\subset P_2$ denote the divisor that is linear equivalent to $X_\infty$ and $X_{\infty,2}\cap X_{\sigma,2}$. Now we consider the invariants of
\[
\left(P_2\times_{P_1}P_2,X_{\infty, 21}+X_{\sigma,21}+X_{\infty,22}+X_{\sigma,22}\right)
\]
 to reduce the number of relative markings by $2$, where $X_{\infty,2i}$ and $X_{\sigma,2i}$ are the divisors $X_{\infty,2}$ and $X_{\sigma,2}$ from the $i$-th copy of $P_2$ in the fiber product $P_2\times_{P_1}P_2$.

  In general, we construct a rank $n$ bundle $P_n$ which is an iterative $\mathbb P^1$-bundle over $X$. And consider the invariants of the fiber product of $P_n$'s relative to its infinity and $\sigma$ divisors.
  }\fi
 

We state the above ideas more precisely on the level of virtual cycles in the following section.

\subsection{Main theorems}
Let $X$ be a smooth projective variety and $D$ be a smooth nef divisor of $X$. We consider a $\mathbb P^1$-bundle 
\[
P:=\mathbb P(\mathcal O_X(-D)\oplus \mathcal O_X)
\]
and let $X_{\infty}$ be the divisor at infinity and $X_\sigma$ be the graph of the section $\sigma$ of $\mathcal O_X(D)$ that defines $D$. The divisor $X_\sigma$ is a divisor linearly equivalent to $X_\infty$ and $X_\sigma \cap X_\infty=D$. 

 Let $\overline{M}_{0,\vec k,n_0,\beta}(X,D)$ be the moduli space of genus zero relative stable maps to $(X,D)$ with degree $\beta\in H_2(X)$, $(n+1)$ relative markings with contact orders
 \[
 \vec k=(k_0,k_1,\ldots,k_n)\in (\mathbb Z_{> 0})^{n+1},
 \]
 and $n_0$ interior markings.

On the other hand, we consider the moduli space 
\[
\overline{M}_{0,(k),n_0,\beta+\sum_{i=1}^nk_if}\left(P,X_{\infty}+X_{\sigma}\right)
\]
of genus zero orbifold stable maps to the multi-root stack associated to the pair $\left(P,X_{\infty}+X_{\sigma}\right)$ with degree $\beta+\sum_{i=1}^nk_if\in H_2(P)$, $(n+1)$  markings with contact orders 
\[
(0,0),(k_1,k_1),(k_2,k_2),\ldots,(k_n,k_n)
\]
 and $n_0$ additional interior markings, as defined in \cite{TY20c}*{Definition 20}.

Relative Gromov--Witten cycles can be pushed forward to $A_*(\overline{M}_{0,n+n_0+1,\beta}(X)\times_{X^{n}} D^{n})$ via the following  diagram (see, for example, \cite{FWY}*{Section 5}), where we omit the topological data for the moduli spaces:

\begin{tikzcd}
\overline{M}(X,D)
\arrow[bend left]{drr}{F}
\arrow[bend right,swap]{ddr}{\on{ev}_{(1,\ldots,n)}}
\arrow[dashed]{dr}[description]{\pi_{\on{rel}}} & & \\
& \overline{M}(X)\times_{X^{n}} D^{n} \arrow{r}{} \arrow{d}[swap]{\on{ev}_{(1,\ldots,n)}}
& \overline{M}(X) \arrow{d}{\on{ev}_{(1,\ldots,n)}} \\
& D^{n} \arrow[swap]{r}{\iota}
& X^{n}.
\end{tikzcd}

There is a similar diagram for $\overline{M}_{0,(k),n_0,\beta+\sum_{i=1}^nk_if}\left(P,X_{\infty}+X_{\sigma}\right)$:

\begin{tikzcd}
\overline{M}(P, X_{\infty}+X_{\sigma})
\arrow[bend left]{drr}{F}
\arrow[bend right,swap]{ddr}{\on{ev}_{(1,\ldots,n)}}
\arrow[dashed]{dr}[description]{u} & & \\
& \overline{M}(X)\times_{X^{n}} D^{n} \arrow{r}{} \arrow{d}[swap]{\on{ev}_{(1,\ldots,n)}}
& \overline{M}(X) \arrow{d}{\on{ev}_{(1,\ldots,n)}} \\
& D^{n} \arrow[swap]{r}{\iota}
& X^{n},
\end{tikzcd}
\\where $F$ is induced by the projection $P\rightarrow X$ and forgets the orbifold structures (root constructions) on $P$. For the marking $p_0$, we have insertions from $H^*(X)$.

We have an identity between the genus zero relative Gromov--Witten theory of $(X,D)$ with $(n+1)$-relative markings and the genus zero orbifold Gromov--Witten theory of $(P,X_{\infty}+X_{\sigma})$ with $n$-orbifold markings:

\begin{theorem}\label{thm-main-1}
  Given a smooth projective variety $X$ with a smooth nef divisor $D$, genus zero relative invariants of $(X,D)$ satisfy the following identity:
        \begin{equation}\label{iden-main}
          \begin{split}
& \pi_{\on{rel},*}\left([\overline{M}_{0,\vec k,n_0,\beta}(X,D)]^{\on{vir}}\right)\\
= &(-1)^{k_0} u_*\left( [\overline{M}_{0,(k),n_0,\beta+ (\sum_{i=1}^nk_i)f}(P, X_{\infty}+X_{\sigma})]^{\on{vir}}\cap\on{ev}_0^*(X_\infty-\pi^*D)\right)\\       
\in        &  A_*(\overline{M}_{0,n+n_0+1,\beta}(X)\times_{X^{n}} D^{n}),
 \end{split}
        \end{equation}
        where
the RHS are the virtual cycles for the moduli spaces of genus zero orbifold stable maps to the pair $(P,X_{\infty}+ X_{\sigma})$ with contact orders 
\[
(0,0), (k_1,k_1),\ldots, (k_n,k_n),
\]
as defined in \cite{TY20c}*{Definition 20}. 
\end{theorem}



\begin{remark}
    The state space for the last $n$ markings are naturally $H^*(D)$ on both sides of (\ref{iden-main}). For the relative marking $p_0$ on the LHS, the insertion is from the ambient space $\iota^*H^*(X)\subset H^*(D)$. The marking $p_0$ on the RHS is an interior marking with insertion coming from $H^*(X)$ where an element of $H^*(X)$ is viewed as an element of $H^*(P)$ via the pullback by $\pi:P\rightarrow X$.
\end{remark}

\begin{remark}
    In \cite{BNR22} and \cite{BNR24}, Battistella--Nabijou--Ranganathan also considered the so-called naive invariants of an snc pair. In genus zero, these naive invariants are the same as the orbifold invariants by \cite{BNR22}. So, we can replace orbifold invariants by naive invariants in Theorem \ref{thm-main-1}.  
\end{remark}



\begin{remark}
     We consider Theorem \ref{thm-main-1} as a generalization of the local-relative correspondence beyond the maximal contacts.  
     For maximal contact invariants, we do not need the fiber class to fix any contact orders because the maximal contact $D\cdot\beta$ is already fixed when $\beta$ is fixed. 

     For relative invariants with contact orders $k_0,\ldots,k_n$, the local theory of $\mathcal O_X(-D)$ does not contain the discrete data to encode these contact orders.
      By considering the invariants of $P$, we can have a choice of the fiber class and choices of contact orders along $X_\infty$ and $X_\sigma$ at markings $p_i$ to encode some additional information.   When we fix a fiber class $\sum_{i=1}^n k_i f$, the remaining contact order $k_0=D\cdot\beta-\sum_{i=1}^n k_i$ is also fixed when $\beta$ is fixed. In this way, we reduce the number of relative markings by $1$. The contact orders $k_i$, $i=1,\ldots,n$, to $D$ are recorded as the contact orders $(k_i,k_i)$, $i=1,\ldots,n$ to $X_\infty+X_\sigma$ and $X_\infty\cap X_\sigma=D$.
 \end{remark}

Now we can consider the Gromov--Witten theory of multi-root stacks defined in \cite{TY20c} for an snc pair $(X,D)$, where $X$ is a smooth projective variety and $D=D_1+\ldots+D_m$ is snc. When all $D_i$ are nef for all $i\in \{1,\ldots,m\}$, there is a local-orbifold correspondence (\cite{TY20b} and \cite{BNTY}) that equals the genus zero local Gromov--Witten theory of the bundle $\bigoplus_{i=1}^m\mathcal O_X(-D_i)$ with the genus zero orbifold Gromov--Witten theory of $(X,D)$ with maximal contacts. The local-orbifold correspondence of \cite{BNTY} follows from either the local-orbifold correspondence for smooth orbifold pairs or the product formula for multi-root stacks \cite{BNTY}*{Theorem B} and \cite{BNR22}*{Theorem 2.5}. 

To state this result for the snc pair $(X,D)$ more precisely, we need to introduce some notation here.
\begin{itemize}
    \item Let
\[
\mathcal X_{m}:=X_{(D_1,r_1),\ldots, (D_{m},r_{m})}
\]
be a multi-root stack of $X$ along $D_i$ with sufficiently large roots $r_i$, for $i=1,\ldots,m$. Then $\mathcal X_m$ is a $r_m$-th root stack of $\mathcal X_{m-1}$ along $\mathcal D_m$, where $\mathcal D_m$ is the pullback of $D_m$ from $X$. By \cite{TY20c}*{Theorem 9}, the genus zero Gromov--Witten invariants of $\mathcal X_m$ does not depend on $r_i$ when $r_i$ is sufficiently large. The large $\vec r$-limits of these invariants are defined as the orbifold invariants of the snc pair $(X,D)$ in \cite{TY20c}*{Section 3}.
\item Let $\mathcal P_m$ be the $\mathbb P^1$-bundle 
\[
\pi_m: \mathcal P_m:=\mathbb P(\mathcal O_{\mathcal X_{m-1}}(-\mathcal D_m)\oplus \mathcal O_{\mathcal X_{m-1}})\rightarrow \mathcal X_{m-1}.
\]
\item Let $\mathcal X_{\infty,m}$ be the divisor of $\mathcal P_m$ at infinity and $\mathcal X_{\sigma,m}$ be the graph of the section $\sigma$ of $\mathcal O_{\mathcal X_{m-1}}(\mathcal D_m)$ defining $\mathcal D_m$. The divisor $\mathcal X_{\sigma,m}$ is a divisor linearly equivalent to $\mathcal X_{\infty,m}$ and $\mathcal X_{\sigma,m} \cap \mathcal X_{\infty,m}=\mathcal D_m$.
\item Let $P_m$ be the coarse moduli space of $\mathcal P_m$. This is a $\mathbb P^1$-bundle $\pi_m:P_m\rightarrow X$. Then $\mathcal P_m$ is a multi-root stack of $P_m$ along the divisor $\pi_m^* D_1+\cdots+\pi_m^* D_{m-1}$.
\end{itemize}
 Note that $\mathcal X_{\infty,m}$ and $\mathcal X_{\sigma,m}$ are pullback from the infinity divisor $X_{\infty,m}$ and the $\sigma$-divisor $X_{\sigma,m}$ in $P_m$.  As $\mathcal P_m$ is a multi-root stack of $P_m$ along the divisor $\pi_m^* D_1+\cdots+\pi_m^* D_{m-1}$, we can also write the pair $\left(\mathcal P_m,\mathcal X_{\infty,m}+\mathcal X_{\sigma,m}\right)$ as $(P_m,\pi_m^* D_1+\cdots+\pi_m^* D_{m-1}+X_{\infty,m}+ X_{\sigma,m})$ and write $(\mathcal X_{m-1}, \mathcal D_m)$ as $(X,D)$.

 Let $\overline{M}_{0,\{\vec k^j\}_{j=1}^m,n_0,\beta}(X,D)=\overline{M}_{0,\{\vec k^j\}_{j=1}^m,n_0,\beta}(\mathcal X_{m-1},\mathcal D_m)$ be the moduli space of genus zero relative stable maps to $(X,D)$ with degree $\beta\in H_2(X)$, $(n+1)$-markings with contact orders
 \[
 \vec k^j=(k_0^j,k_1^j,\ldots,k_n^j)\in (\mathbb Z_{> 0})^{n+1}, \text{ for } j=1,\ldots,m,
 \]
 and $n_0$ interior markings.
Here we follow the notation in \cite{TY20c} where the superscript $j$ is indexing the irreducible components of the divisors, not representing powers of the numbers.

On the other hand, we consider the moduli space 
\[
\overline{M}_{0,(k^j)_{j=1}^m,n_0,\beta+\sum_{i=1}^nk_i^mf_m}\left(\mathcal P_m,\mathcal X_{\infty,m}+\mathcal X_{\sigma,m}\right)
\]
of genus zero orbifold stable maps to the multi-root stack associated to the pair $\left(\mathcal P_m,\mathcal X_{\infty,m}+\mathcal X_{\sigma,m}\right)$ with degree $\beta+\sum_{i=1}^nk_i^mf_m\in H_2(\mathcal P_m)$, $i$-th marking carries contact orders $k_i^j$ with $\pi_m^* D_j$ for $i=1,\ldots,n$. The contact orders with $X_{\infty,m}$ and $X_{\sigma,m}$ are 
\[
(0,0),(k_1^m,k_1^m),(k_2^m,k_2^m),\ldots,(k_n^m,k_n^m),
\]
and $n_0$ additional interior markings.


Similarly to Theorem \ref{thm-main-1}, we consider the following diagrams:

\begin{tikzcd}
\overline{M}(\mathcal X_{m-1},\mathcal D_m)
\arrow[bend left]{drr}{F}
\arrow[bend right,swap]{ddr}{\on{ev}_{(1,\ldots,n)}}
\arrow[dashed]{dr}[description]{\pi_{\on{rel}}} & & \\
& \overline{M}(\mathcal X_{m-1})\times_{\mathcal X_{m-1}^{n}} \mathcal D_m^{n} \arrow{r}{} \arrow{d}[swap]{\on{ev}_{(1,\ldots,n)}}
& \overline{M}(\mathcal X_{m-1}) \arrow{d}{\on{ev}_{(1,\ldots,n)}} \\
& \mathcal D_m^{n} \arrow[swap]{r}{\iota}
& \mathcal X_{m-1}^{n}.
\end{tikzcd}

and

\begin{tikzcd}
\overline{M}(\mathcal P_{m}, \mathcal X_{\infty,m}+\mathcal X_{\sigma,m})
\arrow[bend left]{drr}{F}
\arrow[bend right,swap]{ddr}{\on{ev}_{(1,\ldots,n)}}
\arrow[dashed]{dr}[description]{u} & & \\
& \overline{M}(\mathcal X_{m-1})\times_{\mathcal X_{m-1}^{n}} \mathcal D_m^{n} \arrow{r}{} \arrow{d}[swap]{\on{ev}_{(1,\ldots,n)}}
& \overline{M}(\mathcal X_{m-1}) \arrow{d}{\on{ev}_{(1,\ldots,n)}} \\
& \mathcal D_m^{n} \arrow[swap]{r}{\iota}
& \mathcal X_{m-1}^{n},
\end{tikzcd}
\\
where we slightly abuse the notation by using the same letters $F$, $u$ etc., as in the smooth divisor case.

\begin{theorem}\label{thm-local-orb}
The generalized local-orbifold correspondence for snc pairs holds beyond maximal contacts:
\begin{equation}\label{iden-main-snc}
          \begin{split}
& \pi_{\on{rel},*}\left([\overline{M}_{0,\vec k,n_0,\beta}(\mathcal X_{m-1},\mathcal D_m)]^{\on{vir}}\right)\\
= &(-1)^{k_0} u_*\left( [\overline{M}_{0,(k),n_0,\beta+ (\sum_{i=1}^nk_i)f}(\mathcal P_{m}, \mathcal X_{\infty,m}+\mathcal X_{\sigma,m})]^{\on{vir}}\cap\on{ev}_0^*(\mathcal X_\infty-\pi_m^*\mathcal D_m)\right)\\       
        & \in A_*(\overline{M}_{0,n+n_0+1,\beta}(\mathcal X_{m-1})\times_{\mathcal X_{m-1}^{n}} \mathcal D_m^{n}).
 \end{split}
        \end{equation}
\end{theorem}

Theorem \ref{thm-local-orb} is a generalization of Theorem \ref{thm-main-1} by adding additional root constructions to some snc divisors. The proof is almost identical to the proof of Theorem \ref{thm-main-1}. We will present it in Section \ref{sec:proof-local-orb}. 


Theorem \ref{thm-main-1} reduces the number of relative markings from $(n+1)$ to $n$. To reduce the number of relative markings from $n$ to $(n-1)$ we can repeat the process using Theorem \ref{thm-local-orb}. After applying Theorem \ref{thm-main-1}, we have two divisors $X_{\infty,1}, X_{\sigma,1}\subset P_1:=P$, for each divisor we can construct a $\mathbb P^1$-bundle 
\[
P_2:=\mathbb P(\mathcal O_{P_1}(-X_{\infty,1}\oplus \mathcal O)).
\]
 By the product formula for multi-root stacks \cite{BNTY}*{Theorem 2.1}, we consider the fiber product
\begin{align}\label{P-2-P-2}
(P_2\times_{P_1}P_2,X_{\infty,21}+X_{\sigma, 21}+X_{\infty,22}+X_{\sigma,22}),
\end{align}
where $X_{\infty,2i}$ and $X_{\sigma,2i}$ are the infinity and $\sigma$ divisors of the $i$-th factor of the fiber product $P_2\times_{P_1}P_2$.

Then, by Theorem \ref{thm-main-1} and Theorem \ref{thm-local-orb}, the genus zero relative Gromov--Witten invariant with $(n+1)$-relative markings with contact order $k_0,k_1,\ldots,k_n$ becomes the genus zero Gromov--Witten invariant of the pair (\ref{P-2-P-2}) with contact orders
\[
\vec 0,\vec 0, (k_2,k_2,k_2,k_2),\ldots,(k_n,k_n,k_n,k_n).
\]

To further reduce the number of relative markings to $0$, we repeat this process $(n-1)$-times and conclude
\begin{theorem}\label{thm-main-n}
    Given a smooth pair $(X,D)$ with $D$ nef, the genus zero relative Gromov--Witten invariants of $(X,D)$ with $(n+1)$ relative markings and ambient insertions equal the (absolute/local) Gromov--Witten invariants of a toric bundle over $X$ with rank $2^{n+1}-1$.
\end{theorem}


\begin{remark}
By the mirror theorems of \cite{Brown}, \cite{CCIT15} and \cite{JTY}, genus zero Gromov--Witten invariants of toric bundles can be computed in terms of the genus zero Gromov--Witten invariants of the base. Theorem \ref{thm-main-n} can be used to compute all genus zero relative Gromov--Witten invariants with ambient insertions.  \end{remark}


By \cite{LP}*{Theorem 2}, when $H^*(X)$ is generated by $\on{Pic}(X)$, genus zero absolute Gromov--Witten invariants of $X$ can be reconstructed from one-point invariants. For toric bundles, one-point invariants can be computed via the mirror theorem of \cite{Brown}. Using the virtual localization formula, genus zero Gromov--Witten invariants of toric bundles can be expressed in terms of genus zero descendant Gromov--Witten invariants of the base $X$. The toric bundles appearing here are iterative $\mathbb P^1$-bundles. As discussed in \cite{CGT} and \cite{Koto}, Brown's theorem \cite{Brown} implies that there is a decomposition of the quantum cohomology of these toric bundles in terms of the quantum cohomology of the base.
This gives a reconstruction theorem for genus zero relative Gromov--Witten invariants.
\begin{corollary}\label{cor-reconst}
Suppose $H^*(X)$ is generated by $\on{Pic}(X)$. Then for a smooth nef divisor $D\subset X$, the genus zero primary relative Gromov--Witten invariants of $(X,D)$ with ambient insertions and with non-negative contacts, can be reconstructed from genus zero one-point invariants of $X$.   
\end{corollary}

In general, by \cite{MP}, relative Gromov--Witten invariants of $(X,D)$ can be uniquely and effectively reconstructed from Gromov--Witten invariants of $X$, Gromov--Witten invariants of $D$ and the restriction map $H^*(X,\mathbb Q)\rightarrow H^*(D,\mathbb Q)$. 

\if{
Since all the relative invariants that appear in the degeneration formula can be determined by the $I$-functions, one may expect that there is a relation between the $I$-function of the variety being degenerated and the $I$-functions of the varieties of the degenerated pieces. Such a relation was found in \cite{DKY} in the toric setting and is expected to be true in general. It was not clear from \cite{DKY} that if the relation among $I$-functions is compatible with the degeneration formula as several obstructions were pointed out at the end of \cite{DKY}. The main issue was that the set of invariants determined by the $I$-functions and the set of invariants appearing in the degeneration formulas are different. With Corollary \ref{cor-reconst}, we can indeed expect that the relation among the $I$-functions are compatible with the degeneration (at least under the assumption of Corollary \ref{cor-reconst}). We view such a relation among $I$-functions or quantum D-modules, as a {\em quantum Mayer–Vietoris theorem}. There is a Mayer-Vietoris theorem for single Gromov--Witten invariants in \cite{MP}. Here, we mean a quantum Mayer-Vietoris theorem on the structural level.
}\fi

\subsection{Computation of genus zero relative Gromov--Witten invariants}

The main theorems reduce the computation of genus zero relative Gromov--Witten invariants with several relative markings to certain local Gromov--Witten invariants. These local targets are  special cases of toric bundles over $X$, whose structures are well-known. For example, there is a mirror theorem for toric bundles \cite{Brown} and there are Virasoro constraints for toric bundles \cite{CGT}. 

The computation of relative Gromov--Witten invariants is usually more complicated than the computation of local Gromov--Witten invariants. A computation of a generating function of genus zero two-point relative Gromov--Witten invariants of a smooth log Calabi--Yau pair $(X,D)$ was presented in \cite{You22}. The generating function is the proper Landau--Ginzburg potential of the mirror of the log Calabi--Yau pair $(X,D)$. 
The computation in \cite{You22} is highly nontrivial and involves considering extended relative $I$-functions, understanding the extended relative mirror map, computing several kinds of auxiliary invariants, and taking partial derivatives, among other technical steps.

Now, Theorem \ref{thm-main-1} implies that relative invariants with $2$ relative markings satisfy the following identity:
\begin{align*}
&\langle [\iota^*\gamma_0]_{k_0},[\gamma_1]_{k_1}\rangle_{0,2,\beta}^{(X,D)}\\
=&(-1)^{k_0}\langle [\gamma_0\cup(X_\infty-\pi^*D)]_{(0,0)}, [\gamma_1]_{(k_1,k_1)} \rangle_{0,2,\beta+k_1f_1}^{(P,X_{\infty}+X_{\sigma})},
\end{align*}
where $k_0+k_1=D\cdot \beta$, $\gamma_0\in \iota^*H^*(X)$, $\gamma_1\in H^*(D)$, $\iota: D\hookrightarrow X$ is the inclusion.
If we also have $\gamma_1\in \iota^*H^*(X)$, the local-orbifold correspondence with maximal contacts \cite{BNTY}*{Theorem A} implies that
\begin{align*}
&\langle [\iota^*\gamma_0]_{k_0},[\gamma_1]_{k_1}\rangle_{0,2,\beta}^{(X,D)}\\
=&(-1)^{k_0}\langle [\gamma_0\cup(X_\infty-\pi^*D)]_0, [\gamma_1\cup X_{\infty}\cup X_{\infty}]_{0} \rangle_{0,2,\beta+k_1f}^{\mathcal O_{P}(-X_{\infty})\oplus \mathcal O_{P}(-X_{\infty})}.
\end{align*}
Therefore, this reduces two-point relative invariants of $(X,D)$ to local invariants of the bundle $\mathcal O_{P}(-X_{\infty})\oplus \mathcal O_{P}(-X_{\infty})$.

In Section \ref{sec: computation}, we use it to reduce the computation of two-point invariants to the computation of one-point local Gromov--Witten invariants. Then we provide a simpler computation via the mirror theorem for toric bundles \cite{Brown}.

\begin{theorem}[=\cite{You22}*{Theorem 5.1}]\label{thm-lg-potential}
Let $X$ be a smooth projective variety with a smooth nef anticanonical divisor $D$. Let 
\[
W:=\vartheta_1=x^{-1}+x^{-1}\sum_{n=1}^\infty\sum_{\beta: D\cdot \beta=n+1}n\langle [1]_1,[\on{pt}]_{n}\rangle_{0,\beta,2}^{(X,D)}q^\beta
\]
be the mirror proper Landau--Ginzburg potential. Then
\[
W=x^{-1}\exp\left(g(y(q))\right),
\]
where
\[
g(y)=\sum_{\substack{\beta\in \on{NE}(X)\\ D\cdot \beta \geq 2}}\langle [\on{pt}]\psi^{D\cdot \beta-2}\rangle_{0,1,\beta}^Xy^\beta (D\cdot \beta-1)!
\]
and $y=y(q)$ is the inverse of the mirror map (\ref{mirror-map}).

\end{theorem}

\begin{remark}
    By Theorem \ref{thm-main-1}, the invariants in Theorem \ref{thm-lg-potential} are equal to invariants of $(P,X_\infty+X_\sigma)$ with contact order $(1,1)$. Geometrically, one may think of these invariants as counting holomorphic disks in the local Calabi--Yau $K_X$ with Maslov index two. This agrees with the conjecture in \cite{GRZ} that the proper Landau--Ginzburg potential is the open mirror map. Of course, we do not actually consider open Gromov--Witten invariants here. Instead, we take a purely algebro-geometric point of view by considering relative invariants with maximal contacts as counting disks in the complement $P\setminus X_\infty=K_X$. One can expect that the open/closed duality of \cite{chan}, \cite{LLW11}, \cite{CCLT} also holds beyond the toric setting, then the conjecture of \cite{GRZ} will also be true beyond the toric setting.
\end{remark}

\begin{remark}
    In \cite{Wang}*{Theorem 1.2}, the first author showed that these two-point relative invariants of $(X,D)$ coincide with one-point absolute invariants of $P:=\mathbb P(\mathcal O_{X}(-D)\oplus \mathcal O)$. This is a special case as one of the relative markings has contact order $1$. In general, two-point relative invariants of $(X,D)$ are not the same as the one-point absolute invariants of $P$ (in particular, the virtual dimensions of the moduli spaces are not the same.) 
\end{remark}

\begin{remark}
    In \cite{You20}*{Theorem 1.1}, the second author computed the instanton corrections (open Gromov--Witten invariants) of toric Calabi--Yau orbifolds in terms of one-point relative Gromov--Witten invariants of a (partial) compactification. In particular, when the toric Calabi--Yau is the total space of a canonical bundle $K_X$ of a semi-Fano toric manifold, the instanton corrections are in terms of one-point relative Gromov--Witten invariants of $(P,X_{\infty})$. It is not hard to observe that the generating function of one-point relative Gromov--Witten invariants computed in \cite{You20}*{Theorem 1.1} coincides with the expression for the proper Landau--Ginzburg potential in \cite{You22}*{Theorem 1.6} (see \cite{GRZ}*{Theorem 1.3} for the computation for toric del Pezzo surfaces via a different method), see also Theorem \ref{thm-lg-potential}:
    \[
    W=x^{-1}\exp\left(g(y(q))\right),
    \]
    up to sign. This also works only when one of the two relative markings has contact order $1$. Nevertheless, this special case served as an early motivation for this project. 
\end{remark}

\subsection{Future directions}

There are lots of natural questions that arise. In addition to some applications such as the computation of relative invariants, we also would like to mention some questions that are directly related to possible generalizations of Theorem \ref{thm-main-1}.
\begin{itemize}
    \item Is there a generalization to snc pairs to give a version of the local-log correspondence, hence effectively determining log Gromov--Witten invariants? 
    There are attempts to generalize the local-relative correspondence to snc pairs. However, the naive generalizations proposed in \cite{vGGR}*{Conjecture 1.4} and \cite{TY20b}*{Conjecture 1.8} are not true in general, as a counter-example is provided in \cite{NR}*{Theorem X}. With the counter-example, one may believe that there is not a generalization of the local-log correspondence to snc pairs. It may be interesting to investigate the local-log correspondence for snc pairs in our setting (or under mild modifications). 
    \item Is there a generalization to the higher genus relative Gromov--Witten invariants? An identity (possibly complicated) certainly exists via degeneration and virtual localization. It would be interesting to see how to package the localization computation to study the structures of relative Gromov--Witten theory, similar to the work of \cite{BFGW}.
    \item Is it possible to allow invariants with negative contact orders and determine the genus zero relative Gromov--Witten theory of \cite{FWY}?

\end{itemize}

\section*{Acknowledgement}

The authors would like to thank Mark Gross, Navid Nabijou, Dhruv Ranganathan and Longting Wu for helpful discussions. The authors would also like to thank the anonymous referees for their detailed comments, which greatly improved the exposition of this article.

This project has received funding from the European Union’s Horizon 2020 research and innovation programme under the Marie Skłodowska-Curie grant agreement 101025386, and from the European Research Council (ERC) under the MSAG grant agreement 101019465. The first author would particularly like to express his sincere gratitude to Mark Gross, his PhD supervisor at the time, for his supervision and support.

\section{The degeneration formula for orbifolds}

We briefly review the degeneration formula for the orbifold Gromov--Witten theory in \cite{AF}*{Proposition 5.9.1} when the central fiber of the degeneration has two irreducible components. In particular, we will apply the degeneration formula to the orbifold Gromov--Witten theory of multi-root stacks defined in \cite{TY20c}. We also refer to \cite{You24}*{Section 4}.

\begin{defn}[Definition 4.6, \cite{Li1}]
{\emph {An admissible graph}} $\Gamma$ is a graph without edges plus the following data.
\begin{enumerate}
    \item An ordered collection of legs.
    \item An ordered collection of edges decorated by the contact orders.
    \item A function $\g:V(\Gamma)\rightarrow \mathbb Z_{\geq 0}$.
    \item A function $b:V(\Gamma)\rightarrow H_2(X,\mathbb Z)$.
\end{enumerate}
\end{defn}

Let $\mathcal X$ be a multi-root stack of $X$ along an snc divisor $D_1+\cdots+D_m$ with roots given by $\vec r=(r_1,\ldots,r_m)$ such that the roots $r_j$ are sufficiently large for all $j\in\{1,\ldots,m\}$. Consider a degeneration of $\mathcal X$, a one-parameter family with a smooth total space, such that the central fiber has two irreducible components $\mathcal X_1\cup_{\mathcal D} \mathcal X_2$. We denote the specialisation of $D_1+\cdots+D_m$ to each irreducible component $\mathcal X_i$ of the degeneration by $D_{1i}+\cdots+D_{mi}$. We assume that $D+D_{1i}+\cdots+D_{mi}$ is snc for $i=1,2$, where $D$ denotes the underlying smooth variety of $\mathcal D$.


Suppose that there are $l_1,l_2$ legs in $\Gamma_1, \Gamma_2$, respectively. We have an order preserving inclusion $I:[l_1]\rightarrow [l_1+l_2]$, where $[l]$ means a set of $l$ elements (following the notation in \cite{Li1}*{Definition 4.11}). The legs are used to record the distribution of the markings. We may simply mention the associated markings when we discuss the legs.

For an admissible triple $\Gamma=(\Gamma_1,\Gamma_2,I)$, let
\[
\vec{\eta}=(\vec{\eta}^1,\ldots,\vec{\eta}^{|\eta|}),
\]
where
\[
\vec{\eta}^i=(\eta^i_{0},\eta^i_1,\ldots,\eta^i_{m})
\]
and $|\eta|$ is the length of the partition,
such that $\eta^i_0\in \mathbb Q_{>0}$ is used to record the contact order of the marking with respect to the relative divisor $\mathcal D$. Furthermore, $\eta^i_j\in \mathbb Z$, for $j=1,\ldots,m$ is used to record the contact order $\eta^i_j$ with respect to the divisors $D_j$.  Note that the divisors $D_j$ are orbifold divisors. When we say contact order $\eta^i_j$, we actually mean age $\eta^i_j/r_j$ if $\eta^i_j\geq 0$ or $1+\eta^i_j/r_j$ if $\eta^i_j< 0$, for sufficiently large $r_j$. We remark that $\eta^i_0$ is always positive, but $\eta^i_j$, for $j=1,\ldots,m$, can be either positive, negative, or zero. If the $i$-th edge of $\vec{\eta}$ has contact order 
\[
\vec{\eta}^i=(\eta^i_{0},\eta^i_1,\ldots,\eta^i_{m}),
\]
then the $i$-th edge of $\vec{\eta}^\vee$ has contact order
\[
\vec{\eta}^{i,\vee}=(\eta^i_{0},-\eta^i_1,\ldots,-\eta^i_{m}).
\]
Negative signs come from the orbifold pairing in the orbifold degeneration formula.

\begin{remark}\label{rem-ages}
    We briefly recall that the relative-orbifold correspondence of \cite{ACW}, \cite{TY18}, \cite{FWY} states that genus zero relative Gromov--Witten invariants of a smooth pair $(X,D)$ and orbifold invariants of the root stack $X_{D,r}$ coincide for $r$ sufficiently large. Relative markings of the pair $(X,D)$ and orbifold markings of the root stack $X_{D,r}$ are related as follows. The positive contact order $k\in \mathbb Z_{>0}$ corresponds to a small age $k/r$ for $r\gg 1$; the negative contact order $k\in \mathbb Z_{<0}$ corresponds to a large age $1+k/r$. Invariants of mid-age are studied in \cite{You21} and \cite{You24}. A pair of mid-age markings have ages $k_a/r$ and $k_b/r$ such that $k_a, k_b \gg |D\cdot\beta|$ and $k_a+k_b=r+c$ for a constant $c\in \mathbb Z$. By the definition of the orbifold invariants, there can not be just one marking with mid-age. In other words, when the topological data (e.g. genus, curve class, number of markings) are fixed, if there is one marking with mid-age (i.e. with age $k/r$ such that $k\gg |D\cdot\beta|$), then there is at least one more marking with mid-age. This is generalized to an snc pair $(X,D=D_1+\cdots+D_m)$ in \cite{TY20c} and \cite{You24} where genus zero orbifold invariants of the multi-root stack $X_{D,\vec r}$ is independent of $\vec r=(r_1,\ldots, r_m)$ when $r_i$'s are sufficiently large for $i=1,\ldots,m$. These invariants are now called orbifold invariants of the snc pair $(X,D)$.
\end{remark}

\begin{remark}
    The words ``relative" and ``orbifold" are interchangeable in our context. We adopt the following convention. When we deal with smooth divisors, we use the word ``relative" because the invariants are actually relative invariants. When the divisor is snc, we use the word ``orbifold" and the invariants are orbifold invariants in \cite{TY20c}.   
\end{remark}

We define $\overline{M}_{\Gamma}$ by the Cartesian diagram

\begin{equation}\label{cartesian-diagram} \begin{tikzcd}
\overline{M}_\Gamma \arrow{r}{} \arrow[swap]{d}{} & \overline{M}^\bullet_{\Gamma_1}(\mathcal X_1,\mathcal {D})\times \overline{M}^\bullet_{\Gamma_2}(\mathcal X_2,\mathcal {D}) \arrow{d}{} \\%
 \mathcal {\bar{I}D}^{|\eta|} \arrow{r}{\Delta}& (\mathcal {\bar{I}D}\times \mathcal {\bar{I}D})^{|\eta|},
\end{tikzcd}
\end{equation}
where  $\overline{M}^\bullet_{\Gamma_i}(\mathcal X_i,\mathcal {D})$ is the moduli space of relative stable maps to the pair $(\mathcal X_i,\mathcal {D})$. Note that we have been hiding the additional root structure along $D_{ji}$ in $\mathcal X_i$. Let $X_i$ be the underlying variety of $\mathcal X_i$, then $\mathcal X_i$ is a multi-root stack of $X_1$ along $D_{1i}+\cdots+D_{mi}$. When we take roots to be sufficiently large, we will also write the pair $(\mathcal X_i,\mathcal D)$ as $(X_i, D+D_{1i}+\cdots+D_{mi})$. This is what we will use in Section \ref{sec:proof}.

The degeneration formula \cite{AF}*{Proposition 5.9.1} is 
\begin{equation}
    \begin{split}
        & [\overline{M}_{0,n,\beta}(\mathcal X)]^{\on{vir}}\\
        = &\sum \frac{\prod \eta^i_0 }{|\Aut(\Gamma)|}\Delta^!\left( [\overline{M}^\bullet_{\Gamma_1}(\mathcal X_1,\mathcal D)]^{\on{vir}}\times [\overline{M}^\bullet_{\Gamma_2}(\mathcal X_2,\mathcal D)]^{\on{vir}} \right),
    \end{split}
\end{equation}
where $\Delta^!$ is the Gysin map; the sum ranges over all admissible triples $\Gamma=(\Gamma_1,\Gamma_2,I)$; and $|\Aut(\Gamma)|$ is the order of the automorphism group $\Aut(\Gamma)$.

\section{Proof of the main theorem}\label{sec:proof}

Recall that we would like to prove the following identity in Theorem \ref{thm-main-1}:
\begin{equation}
          \begin{split}
          & \pi_{\on{rel},*}\left([\overline{M}_{0,\vec k,n_0,\beta}(X,D)]^{\on{vir}}\right)\\
          =&(-1)^{k_0} u_*\left( [\overline{M}_{0,(k),n_0,\beta+ (\sum_{i=1}^nk_i)f}(P, X_{\infty}+X_{\sigma})]^{\on{vir}}\cap\on{ev}_0^*(X_\infty-\pi^*D)\right)\\       
        \in & A_*(\overline{M}_{0,n+n_0+1,\beta}(X)\times_{X^{n}} D^{n}),
 \end{split}
        \end{equation}
      where the LHS is the virtual cycle for the moduli space of  genus zero relative stable maps to $(X,D)$ with $(n+1)$ relative markings $k_0,k_1,\ldots,k_n$ and $n_0$ interior markings; the RHS is the virtual cycle for the moduli space of genus zero orbifold stable maps of multi-root stacks of $(P, X_{\infty}+X_{\sigma})$ defined in \cite{TY20c}*{Section 3}. The RHS has $(n_0+1)$ interior marking and $n$-orbifold markings with contact orders $(k_i,k_i)$, $i=1,\ldots, n$.

\subsection{Degeneration}\label{sec:deg}

We first fix our notation here.
\begin{itemize}
    \item $P:=\mathbb P(\mathcal O_X(-D)\oplus \mathcal O_X)$;
    \item $X_{\infty}:=\mathbb P(\mathcal O_X(-D))\subset P$ is the divisor that we add in at the infinity to compactify $\mathcal O_{X}(-D)$.
    \item $X_{\sigma}\subset P$, as defined in \cite{FTY}*{Section 3.1}, is the zero locus of the section of $\mathcal O_P(1)$ defined as the graph of a section $\sigma$ of $\mathcal O_X(D)$, where the  vanishing locus of $\sigma$ is the divisor $D$. In other words, the section $\sigma$ defines a homomorphism of sheaves 
    \[f=(\sigma\oplus 1)^*\in \Hom_X(\mathcal O_X(-D)\oplus\cO_X,\cO_X).\]
The section $f$ determines a section 
\[
\tilde{f}\in H^0(P,\cO_P(1)),
\]
by the canonical identifications
\[
\Hom_X(\mathcal O_X(-D)\oplus\cO_X,\cO_X)=H^0(X,(\mathcal O_X(-D)\oplus\cO_X)^*)=H^0(X,\pi_*\cO_P(1))=H^0(P,\cO_P(1)).
\]
Then
\[
X_\sigma:=\tilde{f}^{-1}(0)\cong X.
\]
We remark that the intersection $X_{\infty}\cap X_{\sigma}$ is $D$ and the divisors $[X_\infty]$ and $[X_\sigma]$ are linearly equivalent.
    \item $Y:=\mathbb P(N_{D/X}\oplus \mathcal O_D)=\mathbb P(\mathcal O_{D}(-D)\oplus \mathcal O_D)$. 
    \item $D_0\subset Y$ is the $0$-section of the normal bundle of $D$ and $D_\infty$ is the $\infty$-section of $Y=\mathbb P(\mathcal O_{D}(-D)\oplus \mathcal O)$.
    \item $P_Y:=\mathbb P(\mathcal O_Y(-D_\infty)\oplus \mathcal O)$.
    \item $Y_\infty$ is the divisor that we add in at infinity to compactify $\mathcal O_Y(-D_\infty)$. 
    \item $Y_\sigma\subset P_Y$ is defined similarly to $X_\sigma$. In particular, $Y_\sigma$ is linear equivalent to $Y_\infty$ and $Y_\sigma\cap Y_\infty=D$, where the intersection is the $\infty$-section of the bundle $Y_\infty$ (which is also the $\infty$-section of the bundle $Y_\sigma$).
\end{itemize}
      
We consider the degeneration in \cite{Wang}*{Section 4}. 
We first consider the degeneration to the normal cone of $D$ in $X$. Let 
\[
\mathfrak X=\on{Bl}_{D\times\{0\}}(X\times \mathbb A^1)\rightarrow \mathbb A^1
\]
and let $\mathfrak D$ be the strict transform of $D\times \mathbb A^1$ in $\mathfrak X$. Then the generic fiber of $\mathfrak X\rightarrow \mathbb A^1$ is $X$ and the central fiber is 
\[
X\cup_D Y,
\]
where we consider $D\subset Y$ as the $0$-section $D_0$ of the normal bundle of $D$.

Now we consider 
\begin{align}\label{deg}
\mathfrak p: \mathfrak P:=\mathbb P(\mathcal O_{\mathfrak X}(-\mathfrak D)\oplus \mathcal O) \rightarrow \mathbb A^1.
\end{align}
For the degeneration $\mathfrak p: \mathfrak P\rightarrow \mathbb A^1$, the generic fiber $\mathfrak p^{-1}(t)$ is $P$ and the central fiber $\mathfrak p^{-1}(0)$ is 
\[
(X\times \mathbb P^1)\cup_{D\times \mathbb P^1} P_Y.
\]
Therefore, we degenerate the pair $(P,X_\infty+X_\sigma)$ into
\[
(X\times \mathbb P^1, (D\times \mathbb P^1)\cup (X\times \infty)\cup (X\times \on{pt})) \text{ and } (P_Y, (D\times \mathbb P^1)\cup Y_\infty\cup Y_\sigma),
\]
where 
\begin{itemize}
\item $D\times \mathbb P^1$ is the common divisor to $(X\times \mathbb P^1)$ and $P_Y$.
\item $X_\infty$ degenerates to $(X\times \infty)\cup_{D\times \infty} Y_\infty$; $X_\sigma$ degenerates to $(X\times \on{pt})\cup_{D\times \on{pt}} Y_\sigma$. 
\item $Y_\infty \cap Y_\sigma=D$ which is the $\infty$-section of the bundle $Y$. So it is away from the divisor $D \times \mathbb P^1=(X\times \mathbb P^1)\cap P_Y$. And the point $\on{pt}$ is away from $\infty\in \mathbb P^1$.
\end{itemize}

As we consider the divisors $X_\infty$ and $X_\sigma$ as orbifold divisors, the invariants that we consider here are orbifold invariants of multi-root stacks in \cite{TY20c}. Observe that the central fiber of the degeneration has two irreducible components, so we are able to use the degeneration formula for the orbifold Gromov--Witten theory \cite{AF}*{Proposition 5.9.1}.


As the invariants are genus zero, the degeneration graphs are trees (no loops) with vertices over $X\times \mathbb P^1$ and $P_Y$, and edges connecting vertices over $X\times\mathbb P^1$ to vertices over $P_Y$. 
\begin{itemize}
    \item Let $V$ be the set of vertices;
    \item let $E$ be the set of edges.
\end{itemize}
Then we have the following relation:
\begin{align*}
    |V|-|E|=1.
\end{align*}

Recall that 
\[
\overline{M}_{0,(k),n_0,\beta+\sum_{i=1}^nk_if}\left(P,X_{\infty}+X_{\sigma}\right)
\]
is  the moduli space of genus zero orbifold stable maps to the multi-root stack associated to the pair $\left(P,X_{\infty}+X_{\sigma}\right)$ with degree $\beta+\sum_{i=1}^nk_if\in H_2(P)$ and $(n+1)$-markings with contact orders 
\[
(0,0),(k_1,k_1),(k_2,k_2),\ldots,(k_n,k_n)
\]
 and $n_0$ additional interior markings.
The degeneration formula expresses \[
u_*\left( [\overline{M}_{0,(k),n_0,\beta+ (\sum_{i=1}^nk_i)f}(P, X_{\infty}+X_{\sigma})]^{\on{vir}}\cap\on{ev}_0^*(X_\infty-\pi^*D)\right)
\]
in terms of the sum
\begin{equation}\label{deg-formula}
    \begin{split}
        \sum \frac{\prod_{i}\eta^i_0}{|\Aut(\Gamma)|}u_{\Gamma,*}\Delta^!\left([\overline{M}^\bullet_{\Gamma_1}((\mathcal X_1,\mathcal D_1)]^{\on{vir}}\times [\overline{M}^\bullet_{\Gamma_2}(\mathcal X_2,\mathcal D_2)]^{\on{vir}}\right)\cap\on{ev}_0^*(X_\infty-\pi^*D),
    \end{split}
\end{equation}
where 
\begin{itemize}
\item the sum ranges over all admissible triples $\Gamma=(\Gamma_1,\Gamma_2,I)$.
\item $u: \overline{M}_{0,(k),n_0,\beta+ (\sum_{i=1}^nk_i)f}(P, X_{\infty}+X_{\sigma}) \rightarrow \overline{M}_{0,n+n_0,\beta}(X)\times_{X^n}D^n$ is the forgetful map.  Similarly, $u_\Gamma: \overline{M}_\Gamma\rightarrow \overline{M}_{0,n+n_0,\beta}(X)\times_{X^n}D^n$ is the forgetful map. We recall that $\overline{M}_\Gamma$ is defined by the Cartesian diagram (\ref{cartesian-diagram}).
\item The blow-down map $\mathfrak X\rightarrow X\times \mathbb A^1$ induces a natural map $$\mathfrak p:\mathfrak P=\mathbb P(\mathcal O_{\mathfrak X}(-\mathfrak D)\oplus \mathcal O) \rightarrow P\times \mathbb A^1.$$ The class $X_\infty-\pi^*D\in H^2(P)$ is lifted by pull-back via the first factor of the map $\mathfrak p$. We use $\on{ev}_0^*(X_\infty-\pi^*D)$ in (\ref{deg-formula}) to denote the constraint that is obtained by lifting $\on{ev}_0^*(X_\infty-\pi^*D)$ via this pullback and then restricted to the central fiber.
    \item $(\mathcal X_1,\mathcal D_1)$ stands for the pair $(X\times \mathbb P^1, (D\times \mathbb P^1)\cup (X\times \infty)\cup (X\times \on{pt}))$, where $D\times \mathbb P^1$ is the relative divisor and $X\times \infty$ and $X\times \on{pt}$ are orbifold divisors, where $\on{pt}$ is a point away from $\infty\in \mathbb P^1$. For vertices in $\Gamma_1$, we will refer to them as vertices on the $\Gamma_1$-side or vertices over $(X\times \mathbb P^1)$.
    \item $(\mathcal X_2,\mathcal D_2)$ stands for the pair $(P_Y, (D\times \mathbb P^1)\cup Y_\infty\cup Y_\sigma)$, where $D\times \mathbb P^1$ is the relative divisor and $Y_\infty$ and $Y_\sigma$ are orbifold divisors. For vertices in $\Gamma_2$, we will refer to them as vertices on the $\Gamma_2$-side or vertices over $P_Y$.
    \item Taking into account the orbifold conditions along $X_\infty$ and $X_\sigma$, 
the common divisor of $\mathcal X_1$ and $\mathcal X_2$ is actually a root stack 
of $D\times \mathbb P^1$ along $D\times \infty$ and $D\times \on{pt}$. 
When matching the edges of $\Gamma_1$ and $\Gamma_2$, the ages along 
$D\times \infty$ and $D\times \on{pt}$ should be inverse to each other 
on the $\Gamma_1$-side and on the $\Gamma_2$-side.

More precisely, if the $i$-th edge of $\vec \eta$ has contact order (age)
\[
\vec \eta^i=(\eta^i_0,\eta^i_1,\eta^i_2)
\]
along the divisors 
$(D\times \mathbb P^1)\cup (X\times \infty)\cup (X\times 0)$, 
then the $i$-th edge of $\vec \eta^\vee$ has contact order (age)
\[
\vec \eta^{i,\vee}=(\eta^i_0,-\eta^i_1,-\eta^i_2)
\]
along the divisors 
$(D\times \mathbb P^1)\cup Y_\infty\cup Y_\sigma$.

For each edge $e_i$, we obtain evaluation maps
\[
ev^1_{e_i} : M_{\Gamma_1} \longrightarrow 
\mathcal I(D\times \mathbb P^1),
\qquad
ev^2_{e_i} : M_{\Gamma_2} \longrightarrow 
\mathcal I(D\times \mathbb P^1),
\]
to the inertia stack of the root stack over $D\times \mathbb P^1$. 
Let $ev_\eta^1$ and $ev_\eta^2$ denote the products of these evaluation maps 
over all edges.

The moduli space associated with the admissible triple $\Gamma$ is defined 
as the fiber product
\[
M_\Gamma
=
M_{\Gamma_1}
\times_{\mathcal I(D\times \mathbb P^1)^{|\eta|}}
M_{\Gamma_2}.
\]
Its virtual fundamental class is obtained by the refined Gysin pullback 
along the diagonal
\[
[M_\Gamma]^{\mathrm{vir}}
=
\Delta^!
\big(
[M_{\Gamma_1}]^{\mathrm{vir}}
\times
[M_{\Gamma_2}]^{\mathrm{vir}}
\big),
\]
where $\Delta$ denotes the diagonal morphism of 
$\mathcal I(D\times \mathbb P^1)^{|\eta|}$. 
In particular, the matching of insertions along the edges is implemented 
entirely through this refined Gysin pullback, and no decomposition of the 
diagonal in Chow is used.

Because the inertia stack of the root stack over $D\times \mathbb P^1$ 
has a product structure, the evaluation maps factor through the projections 
to $D$ and to the orbifold pair $(\mathbb P^1,0+\infty)$. 
This allows us to describe the relevant state space in terms of the state 
space for the relative Gromov--Witten theory of $(\mathbb P^1,0+\infty)$ 
defined in \cite{FWY}*{Section 7.1}:
\[
\mathfrak H (\mathbb P^1,0+\infty)=\oplus_{i,j\in \mathbb Z} \mathfrak H_{i,j},
\]
where 
\begin{equation}
\mathfrak H_{i,j}:=
\left\{
\begin{array}{cc}
H^*(\mathbb P^1) & (i,j)=(0,0);\\
H^*([\on{pt}]) & (i,j)\neq(0,0)\ \text{and}\ ij=0;\\
0 & ij\neq 0.
\end{array}
\right.
\end{equation}

We remark that the degree of the cohomology classes on the twisted sectors 
are shifted:
\begin{equation}
\deg([1]_{(i,0)})=
\left\{
\begin{array}{cc}
0 & i\ge 0,\\
2 & i<0,
\end{array}
\right.
\end{equation}
and
\begin{equation}
\deg([1]_{(0,j)})=
\left\{
\begin{array}{cc}
0 & j\ge 0,\\
2 & j<0.
\end{array}
\right.
\end{equation}
\end{itemize} 

Let $\{p_i\}_{i=0}^{n+n_0}$ be the set of markings for orbifold stable maps to $(P,X_\infty+X_\sigma)$. We remark that the markings $\{p_i\}_{i=1}^n$ are orbifold markings with respect to both $X_\infty$ and $X_\sigma$. These orbifold markings are mapped to $X_\infty\cap X_\sigma=D$. Under the degeneration (\ref{deg}) these orbifold markings are distributed to the $P_Y$-side. 

Now we consider the marking $p_0$. By assumption $p_0$ maps to $X_0$, so it can be distributed to either the $(X\times \mathbb P^1)$-side or the $P_Y$-side. Later, when we analyse the degeneration graph, we will show that if $p_0$ is on the $(X\times \mathbb P^1)$-side, the virtual cycle will vanish. As a result, we will show that $p_0$ will also be distributed to the $P_Y$-side.

In general, in the orbifold degeneration formula, there are no restrictions to the ages along the orbifold divisors as long as the orbifold pairing condition is satisfied for each edge. For multi-root stacks, that means the ages $\eta^i_j$ are not necessarily small ages or large ages and there is a possibility that orbifold invariants with mid-ages will appear. We refer to Remark \ref{rem-ages} for a brief explanation of the meaning of small ages, large ages, and mid-ages. We claim that mid-age invariants will not appear in our setting in genus zero.

\begin{lemma}\label{lemma-mid-age}
No edges with mid-ages with respect to $X\times \infty, X\times \on{pt}, Y_\infty$ or  $Y_\sigma$ will appear in the degeneration formula.    
\end{lemma}

\begin{proof}
Recall that $X_\infty\subset P$ degenerated into $(X\times \infty)\cup Y_\infty$ and $X_\sigma\subset P$ degenerated into $(X\times \on{pt})\cup Y_\sigma$. We will show that there are no edges that has mid-ages along $(X\times \infty)$ and $Y_\infty$. The proof for $(X\times \on{pt})$ and $Y_\sigma$ is identical.

    If one of the vertices is connected to an edge that carries a mid-age along $(X\times \infty)$ and $Y_\infty$ , then, as we explained in Remark \ref{rem-ages}, there should be at least one more edge that carries a mid-age associated with this vertex. We consider a chain of vertices and edges that are connected by edges with mid-ages along $(X\times \infty)$ and $Y_\infty$. Recall that the degeneration graphs are trees (no loops) in genus zero, then at the end of the chain, there will be a vertex that only contains one mid-age marking. Then it is a contradiction.
\end{proof} 


\subsection{Virtual cycles in the degeneration formula}

\subsubsection{Virtual cycles on the $(X\times \mathbb P^1)$-side}
The virtual cycles
\[
[\overline{M}^\bullet_{\Gamma_1}(\mathcal X_1,\mathcal D_1)]^{\on{vir}}
\]
on the $(X\times \mathbb P^1)$-side are virtual cycles for the moduli spaces of stable maps to the pair $(X\times \mathbb P^1, (D\times \mathbb P^1)\cup (X\times \infty)\cup (X\times \on{pt}))$, where the divisor $D\times \mathbb P^1$ is a relative divisor and the divisors $(X\times \infty)$ and $(X\times \on{pt})$ are orbifold divisors. By the relative-orbifold correspondence for orbifolds in \cite{TY20c}*{Theorem 9}, the virtual cycles here are equal to the virtual cycles for the orbifold Gromov--Witten theory associated with the pair $(X\times \mathbb P^1, (D\times \mathbb P^1)\cup (X\times \infty)\cup (X\times \on{pt}))$ defined in \cite{TY20c}*{Definition 20}. By the product formula for orbifolds \cite{AJT16}*{Theorem 4.3, Theorem 4.6}, these virtual cycles are the product of the Gromov--Witten virtual cycles of $(X,D)$ and $(\mathbb P^1,0\cup\infty)$:
\[
[\overline{M}^\bullet_{\Gamma_1}(X,D)]^{\on{vir}} \text{ and } [\overline{M}^\bullet_{\Gamma_1}(\mathbb P^1,0+\infty)]^{\on{vir}}.
\]



Let $v$ be a vertex over $(X\times \mathbb P^1)$, let $E(v)$ be the set of edges attached to $v$. For each vertex $v$ over $X\times \mathbb P^1$, we push forward
\[
[\overline{M}_{\Gamma_1,v}(\mathbb P^1,0+\infty)]^{\on{vir}}\cap \left(\cup_{e\in E(v)}\on{ev}_e^*([\delta_{e}])\right),
\]
where $[\delta_{e}]\in \mathfrak H (\mathbb P^1,0+\infty)$.

\begin{lemma}\label{lemma-X-P-vanish}
     If $\deg([\delta_{e}])=0$ for all $e\in E(v)$ and $p_0$ is not associated with $v$, then the pushforward of 
     \[
[\overline{M}_{\Gamma_1,v}(\mathbb P^1,0+\infty)]^{\on{vir}}\cap \left(\cup_{e\in E(v)}\on{ev}_e^*([\delta_{e}])\right),
\]
to $\overline{M}_{0,n_v}$ vanishes.
\end{lemma}
\begin{proof}
Observe that the markings $\{p_i\}_{i=1}^n$ are distributed to $\Gamma_2$ and the markings in $\Gamma_1$ are from the edges.     

For the relative Gromov--Witten class of $\overline{M}^\bullet_{\Gamma_1}(\mathbb P^1,0+\infty)$ with non-negative contact orders to be non-zero, there needs to be at least one insertion with $[\on{pt}]$. It can be easily proved by the degeneration formula or localization, as observed in \cite{OP06}*{Proposition 2.5}. The localization computation is special case of the computation in Lemma \ref{lemma-rubber} (We also state this lemma here because this is what we need on the $X\times \mathbb P^1$-side).

However, the condition $\deg([\delta_{e}])=0$ means that the insertion $\delta_e$ is an identity class and the contact orders with the divisors are zero or positive. Therefore, the pushforward vanishes.
\end{proof}

\subsubsection{Virtual cycles on the $P_Y$-side}
Now we examine the virtual cycles on the $P_Y$-side:
\[
[\overline{M}^\bullet_{\Gamma_2}(\mathcal X_2,\mathcal D_2)]^{\on{vir}}.
\]

Observe that $Y_\sigma$ is linearly equivalent to $Y_\infty$. The divisor $D\times \mathbb P^1$ is a relative divisor, by the relative-orbifold correspondence for orbifolds in \cite{TY20c}*{Theorem 9}, we can also consider it as an orbifold divisor. Therefore, the virtual cycle becomes the virtual cycle for the moduli space of orbifold stable maps to the pair $(P_Y, (D\times \mathbb P^1)\cup Y_\infty\cup Y_\sigma)$. 

We first need an orbifold version of \cite{FTY}*{Theorem 3.6} in order to analyse the virtual cycles on the $P_Y$-side later. Let $\mathcal X$ be a smooth Deligne--Mumford stack. Let $\mathcal D$ be a smooth divisor of $\mathcal X$ defined by the zero locus of a section of a convex line bundle $\mathcal O_{\mathcal X}(\mathcal D)$. The root stack $\mathcal X_{\mathcal D,r}$ is a hypersurface of the $\mathbb P^1[r]$-bundle $\mathcal P_{\mathcal X_\infty,r}\rightarrow X$, where $\mathcal P:=\mathbb P(\mathcal O_{\mathcal X}(-\mathcal D)\oplus \mathcal O)$ and $\mathcal P_{X_\infty,r}$ is the $r$-th root stack of $\mathcal P$ along the infinity divisor $\mathcal X_\infty$. The hypersurface is the zero locus of a section of $p^*\mathcal O_{\mathcal P}(1)$, where  $p:\mathcal P_{\mathcal X_\infty,r}\rightarrow \mathcal P$ is the projection that forgets the $r$-th root construction along $\mathcal X_\infty$.
\begin{lemma}\label{lemma-FTY-orb}
    Let $\mathcal X$ be a smooth Deligne--Mumford stack. Let $\mathcal D$ be a smooth divisor of $\mathcal X$ defined by the zero locus of a section of a convex line bundle $\mathcal O_{\mathcal X}(\mathcal D)$. Let $\mathcal X_r$ be the gerbe of $r$-th roots of $\mathcal O_{\mathcal X}(\mathcal D)$ over $\mathcal X$. Then we have
    \[
\pi_{1,*}\left[\overline{M}_{0,n,\beta}(\mathcal X_{\mathcal D,r})\right]^{\on{vir}}=\pi_{2,*}\left(\left[\overline{M}_{0,n,\beta}(\mathcal X_{r})\right]^{\on{vir}}\cap \frac{e_{\mathbb C^\times}(\mathcal O(\mathcal D)_{0,n,\beta})}{e_{\mathbb C^\times}(\mathcal O(\mathcal D/r)_{0,n,\beta})}\right)_{t=0},
    \]
    where 
    \[
    \pi_1:\overline{M}_{0,n,\beta}(\mathcal X_{\mathcal D,r})\rightarrow \overline{M}_{0,n,\beta}(\mathcal X);
    \]
    \[
    \pi_2:\overline{M}_{0,n,\beta}(\mathcal X_{r}) \rightarrow \overline{M}_{0,n,\beta}(\mathcal X)
    \]
    are forgetful maps that forget the orbifold structures on the target; $t$ is the equivariant parameter for the natural $\mathbb C^\times$-action on $\mathcal P$.
\end{lemma}

\begin{proof}
    The proof is a generalization of the proof of \cite{FTY}*{Theorem 3.6} to the orbifold setting. 

Let $i:\overline{M}_{0,n,\beta}(\mathcal X_{\mathcal D,r})\rightarrow \overline{M}_{0,n,\beta}(\mathcal P_{\mathcal X_\infty,r})$ be the map induced by the inclusion $i:\mathcal X_{\mathcal D,r}\rightarrow \mathcal P_{\mathcal X_\infty,r}$.    By the orbifold quantum Lefschetz principle \cite{Tseng}*{Proposition 5.2.3}, the Gromov--Witten class $i_{*}\left[\overline{M}_{0,n,\beta}(\mathcal X_{\mathcal D,r})\right]^{\on{vir}}$ equals the Gromov--Witten class of $(\mathcal P_{\mathcal X_{\infty},r})$ twisted by $(e,p^*\mathcal O_{\mathcal P}(1))$:
    \[
i_{*}\left[\overline{M}_{0,n,\beta}(\mathcal X_{\mathcal D,r})\right]^{\on{vir}}=\left[\overline{M}_{0,n,\beta}(\mathcal P_{\mathcal X_{\infty},r})\right]^{\on{vir}}\cap e((p^*\mathcal O_{\mathcal P}(1))_{0,n,\beta}).
    \]

    
    There is a natural $\mathbb C^\times$-action on the fiber of $\mathcal P$ with weight $-1$ on $\mathcal O_{\mathcal X}(-\mathcal D)$ and acts trivially on $\mathcal O$. It induces a $\mathbb C^\times$-action on $\mathcal P_{\mathcal X_\infty,r}$ and a $\mathbb C^\times$-action on the moduli space $\overline{M}_{0,n,\beta}(\mathcal P_{\mathcal X_{\infty},r})$. 
    The fixed loci of the $\mathbb C^\times$-action can be labeled by decorated bipartite graphs. 

As explained in \cite{FTY}*{Theorem 3.6}, the twisting $e((p^*\mathcal O_{\mathcal P}(1))_{0,n,\beta})$ implies that the localization graphs containing an edge always contribute $0$. The only contribution is the graph $\Gamma_0$ which consists of a single vertex over $p^{-1}(\mathcal X_\infty)$ and no edge. The localization contribution of the $\Gamma_0$ is precisely
\[
\pi_{2,*}\left(\left[\overline{M}_{0,n,\beta}(\mathcal X_{r})\right]^{\on{vir}}\cap \frac{e_{\mathbb C^\times}(\mathcal O(\mathcal D)_{0,n,\beta})}{e_{\mathbb C^*}(\mathcal O(\mathcal D/r)_{0,n,\beta})}\right)_{t=0}.
\]
    
\end{proof}


The following lemma will also be used.
\begin{lemma}\label{lemma-rubber}
    The following cycle
    \begin{align}\label{cycle-rigidification}
        \epsilon_* [\overline M_{0,n,\beta+kf}(Y,D_0+D_\infty)]^{\on{vir}}\in A_*(\overline M_{0,n,\beta+kf}(Y))
    \end{align}    
    vanishes unless $\beta=0\in H_2(D)$ and $n=2$, where $\overline M_{0,n,\beta+kf}(Y,D_0+D_\infty)$ is the moduli space of relative stable maps to $(Y,D_0+D_\infty)$, $f$ is the class of a fiber, and
    \[
    \epsilon: \overline M_{0,n,\beta+kf}(Y,D_0+D_\infty)\rightarrow \overline M_{0,n,\beta+kf}(Y)
    \]
    is the forgetful map.
\end{lemma}


\begin{proof}
This is essentially the rigidification lemma of \cite{MP}*{Lemma 2}. We include the proof here for the completeness. This can be proved via a virtual localization computation. Let $t^\prime$ be the equivariant parameter for the $\mathbb C^\times$-action induced by the natural $\mathbb C^\times$-action that scales the fiber $Y\rightarrow D$. If $\beta\neq 0\in H_2(D)$ or $n>2$, there is at least a stable vertex in the virtual localization formula. In the localization formula, the vertex contribution over $D_\infty$ (and the vertex contribution over $D_0$ is similar) is:
\begin{align}\label{contri-infinity}
\prod_{e\in E(v)}C_{(e,v)}\frac{\prod_{e\in E(\Gamma)}d_e}{t^\prime+\psi_\infty},
\end{align}
where $C_{(e,v)}$ is a constant that does not depend on $t^\prime$ and $\psi_\infty$ is the first Chern class of the cotangent line bundle at the relative divisor $D_\infty$.
This always contributes to negative powers of $t^\prime$. Hence the cycle (\ref{cycle-rigidification}) vanishes.
\end{proof}

\subsection{The degeneration graphs}

We first present a sketch of our analysis of the degeneration graph.

\begin{itemize}

\item Edges: we prove Lemma \ref{lemma-edge}: all edges have ages $0$ with the orbifold divisors $X\times \infty$, $X\times \on{pt}$, $Y_\infty$ and $Y_\sigma$. As a consequence of Lemma \ref{lemma-edge}, we conclude Corollary \ref{cor-no-fiber}: the vertices over $X\times \mathbb P^1$ carry no fiber class of the trivial fibration $X\times \mathbb P^1\rightarrow X$. For each vertex over $X\times \mathbb P^1$, there is exactly one edge whose cohomology weight along the $\mathbb P^1$-direction is $[\on{pt}_{\mathbb P^1}]$, and other edges do not carry point constraints along the $\mathbb P^1$-direction. 

\item For vertices with $p_i$ attached, we prove Lemma \ref{lemma-p-i-p-0}: the markings $p_i$ for $i\in \{1,\ldots,n\}$ and the marking $p_0$ are not attached to the same vertex. 

Then we consider the vertex $v$ over $P_Y$ containing one of $p_i$, for $i=1,\ldots,n$. We prove Lemma \ref{lemma-P-Y-vanish}: there can be exactly one edge and one $p_i$ associated with $v$. The contact order of the edge with $D\times \mathbb P^1$ is $k_i$ and the curve classes are fiber classes of $P_Y\rightarrow D$.

\item Then we analyse other vertices in Section \ref{sec:other-vertices}. The conclusion is that there is only one more vertex $v_0$ on the $P_Y$-side with one edge and the marking $p_0$ is attached to $v_0$.

\item We conclude the analysis of the degeneration graph with Proposition \ref{prop-deg-graph}: The term of the degeneration graph $\Gamma$ in the degeneration formula vanishes unless    there is exactly one vertex over $(X\times \mathbb P^1)$ with $(n+1)$-edges connecting this vertex to the $(n+1)$-vertices over $P_Y$. We denote these vertices over $P_Y$ by $v_i$ for $i=0,\ldots,n$.

\item Then in Section \ref{sec:conclusion}, we conclude that all the virtual cycles over the $P_Y$-side are trivial and all additional interior markings are mapped to $(X\times \mathbb P^1)$. 
Therefore, the virtual cycle for the vertex $v_X$ over the $X\times \mathbb P^1$-side  is the virtual cycle on the LHS of (\ref{iden-main}) and this concludes the proof of Theorem \ref{thm-main-1}.

\end{itemize}

\begin{center}

\begin{tikzpicture}

\fill (0,0) circle (2pt);
\node[left] at (0,0) {$v_X$};

\fill (3,1.5) circle (2pt);
\node[right] at (3,1.5) {$v_0$};
\draw (0,0) -- (3,1.5);

\fill (3,0.8) circle (2pt);
\node[right] at (3,0.8) {$v_1$};
\draw (0,0) -- (3,0.8);

\node at (3,0) {$\vdots$};

\draw[dotted] (0,0) -- (3,0.3);
\draw[dotted] (0,0) -- (3,-0.3);

\fill (3,-0.8) circle (2pt);
\node[right] at (3,-0.8) {$v_{n-1}$};
\draw (0,0) -- (3,-0.8);

\fill (3,-1.5) circle (2pt);
\node[right] at (3,-1.5) {$v_n$};
\draw (0,0) -- (3,-1.5);

\end{tikzpicture}
\end{center}

\subsubsection{Edges}
We first study the contact orders/ages of the edges.

\begin{lemma}\label{lemma-edge}
The term of the degeneration graph $\Gamma$ in the degeneration formula vanishes unless  all the edges in  $\Gamma$ carry contact orders of the form $(\eta_0^i,0,0)$. That is, all the edges have ages (contact orders) $0$ with the orbifold divisors $X\times \infty, X\times \on{pt}, Y_\infty$ and  $Y_\sigma$.
\end{lemma} 
\begin{proof}
Observe that for each vertex over $P_Y$, the total contact orders with $Y_\infty$ and $Y_\sigma$ should be the same; for each vertex over $X\times \mathbb P^1$, the total contact orders with $X\times \infty$ and $X\times \on{pt}$ should be the same. An edge attached to a vertex over $P_Y$ cannot have non-zero contact orders to both $Y_\infty$ and $Y_\sigma$. An edge attached to a vertex over $X\times \mathbb P^1$ cannot have non-zero contact orders to both $X\times \infty$ and $X\times \on{pt}$. We recall that the contact orders at $p_i$ are $(k_i,k_i)$ for $i\in\{1,\ldots,n\}$.

Suppose one of the edges has non-zero contact with one of the divisors, for example $X\times \infty$, there needs to be at least one more edge associated with the same vertex that carries a non-zero contact with $X\times \infty$ or $X\times \on{pt}$, otherwise the total contact order with  $X\times \infty$ and $X\times \on{pt}$ will not be the same for this vertex.

This will form a chain of vertices and edges that have non-zero contact with $X\times \infty$ or $X\times \on{pt}$. At the end of the chain, there will be a vertex with just one edge that has non-zero contact with $X\times \infty$ or $X\times \on{pt}$. Contradiction!     
\end{proof}

\begin{remark}
    The proof of Lemma \ref{lemma-edge} looks very similar to the proof of Lemma \ref{lemma-mid-age}. The difference is that Lemma \ref{lemma-edge} depends on our specific choice of divisors and the contact orders at the markings. While Lemma \ref{lemma-mid-age} works in general.
\end{remark}

The following result is a direct consequence of Lemma \ref{lemma-edge}.
\begin{corollary}\label{cor-no-fiber} 
The vertices over $X\times \mathbb P^1$ carry no fiber class of the trivial fibration $X\times \mathbb P^1\rightarrow X$. For each vertex over $X\times \mathbb P^1$ there is exactly one edge whose cohomology weight along the $\mathbb P^1$-direction is $[\on{pt}_{\mathbb P^1}]$, and other edges do not carry point constraints along the $\mathbb P^1$-direction. 
\end{corollary}

\subsubsection{Vertices with $p_i$ attached.}\label{sec:vertices-p-i}
Let $v_0$ be the vertex that contains the marking $p_0$. At this moment, we do not know if $v_0$ is a vertex over $(X\times \mathbb P^1)$ or over $P_Y$.

\begin{lemma}\label{lemma-p-i-p-0}
The term of the degeneration graph $\Gamma$ in the degeneration formula vanishes unless     the markings $p_i$ for $i\in \{1,\ldots,n\}$ are not attached to the vertex $v_0$.
\end{lemma}
\begin{proof}

As $p_i$'s are distributed to the $P_Y$-side, if $v_0$ is a vertex over $(X\times \mathbb P^1)$, then $p_i$'s are not attached to $v_0$. Now we consider the case when $p_0$ is distributed to the $P_Y$-side.

    Suppose that a marking $p_i$, for some $i\in \{1,\ldots,n\}$, is attached to $v_0$. Then by Lemma \ref{lemma-FTY-orb}, the virtual cycle $\overline{M}_{\Gamma_2,v}(\mathcal X_2,\mathcal D_2)$ of $(P_Y, (D\times \mathbb P^1)\cup Y_\infty\cup Y_\sigma)$ becomes a twisted virtual cycle of $(P_Y, (D\times \mathbb P^1)\cup Y_\infty)_{r}$ by $\frac{e(\mathcal O(Y_\infty)_{0,n_v,\beta_v})}{e((Y_\infty/r)_{0,n_v,\beta_v})}$, where $(P_Y, (D\times \mathbb P^1)\cup Y_\infty)_{r}$ is a $\mu_r$-gerbe over the multi-root stack associated with the pair $(P_Y, (D\times \mathbb P^1)\cup Y_\infty)$. 

The orbifold quantum Lefschetz principle of \cite{Tseng}*{Proposition 5.2.3} can be applied to the twist $e(\mathcal O(Y_\infty)_{0,n_v,\beta_v})$ because $\mathcal O(Y_\infty)$ is convex. Therefore, the virtual cycle $\overline{M}_{\Gamma_2,v}(\mathcal X_2,\mathcal D_2)$ becomes the virtual cycle of $(Y_\sigma=Y,D_0\cup D_\infty)_r$ twisted by $\frac{1}{e(D_\infty/r)_{0,n_v,\beta_v}}$, where $p_i$ becomes a relative marking along $D_\infty$ with contact order $k_i$. We remark that $D_0=Y_\sigma\cap (D\times \mathbb P^1)$ and $D_\infty=Y_\sigma\cap Y_\infty$.

On the other hand, we require $p_0$ to map to $Y_0$ which does not intersect $Y_\sigma$. Observe that $i^*X_\infty-D=0$, where $i:Y_\sigma\hookrightarrow P$ is the inclusion. This implies that $v_0$ cannot contain $p_0$. It contradicts the definition of $v_0$. 

One can also prove it via a virtual localization computation, similar to the argument in the proof of Lemma \ref{lemma-P-Y-vanish} below. If we consider the virtual localization formula for the $\mathbb C^\times$-action induced by the natural $\mathbb C^\times$-action of $Y\rightarrow D$, then there will only be negative powers of $t^\prime$, where $t^\prime$ is the equivariant parameter, unless there are two markings and the curve classes are fiber classes. 
\end{proof}


\begin{lemma}\label{lemma-P-Y-vanish}
The term of the degeneration graph $\Gamma$ in the degeneration formula vanishes unless the following holds. For each vertex $v$ over $P_Y$ that contains one of $p_i$, for $i=1,\ldots,n$, there can be exactly one edge and one $p_i$ associated with $v$. The contact order of the edge with $D\times \mathbb P^1$ is $k_i$, and the curve classes are fiber classes of $P_Y\rightarrow D$.
\end{lemma}

\begin{proof}

Recall that the pushforward of the virtual cycle $\overline{M}_{\Gamma_2,v}(\mathcal X_2,\mathcal D_2)$ to the moduli space of stable maps to $Y$ is the pushforward of the virtual cycle for the Gromov--Witten theory of $(Y_\sigma=Y,D_0\cup D_\infty)_r$ twisted by $\frac{1}{e(D_\infty/r)_{0,n_v,\beta_v}}$.

By Lemma \ref{lemma-rubber},   the pushforward of the virtual cycle of $(Y, D_0+D_\infty)$ without a non-negative contact vanishes unless there are two markings and the curve class is a fiber class of the fibration $Y\rightarrow D$.

   Observe that there are two different $\mathbb C^\times$-actions for $P_Y$. One comes from the bundle $P_Y\rightarrow Y$ and the other comes from the bundle $Y\rightarrow D$.
   
   Here we have an extra twist $\frac{1}{e_{\mathbb C^\times}(D_\infty/r)_{0,n_v,\beta_v}}$. The pushforward of the localization contribution of $\frac{1}{e_{\mathbb C^\times}(D_\infty/r)_{0,n_v,\beta_v}}$ was calculated in \cite{TY20c}*{Section 2.2.2}. There will only be non-positive powers of $t^\prime$ in the localization formula when there are no negative contacts (large ages). More precisely, the contribution of a stable vertex $v$ in the degeneration formula is given in \cite{TY20c}*{Equation (2.2)} which in genus zero specializes to:
   \begin{align}\label{eqn:r_dep}
&\left(\prod_{e\in E(v)}C_{(e,v)}\frac{rd_e}{t^\prime+\on{ev}_{e}^*c_1(N_{D/X})-d_e\bar{\psi}_{(e,v)}}\right)\cdot\left(\sum_{i=0}^{\infty}(t^\prime/r)^{-1+|E(v)|-i}c_i(-R^\bullet\pi_*\mathcal L)\right)\\
\notag =& (t^\prime)^{-1}\left(\prod_{e\in E(v)}C_{(e,v)}\frac{d_e}{1+(\on{ev}_{e}^*c_1(N_{D/X})-d_e\bar{\psi}_{(e,v)})/t^\prime}\right)\cdot\left(\sum_{i=0}^{\infty}(t^\prime)^{-i}(r)^{i+1}c_i(-R^\bullet\pi_*\mathcal L)\right),
\end{align}
where 
\begin{itemize}
    \item $C_{(e,v)}$ is a constant that does not involve $t^\prime$.
    \item $E(v)$ is the set of edges associated with the vertex $v$.
    \item Set $L=N_{D/X}$. Then
    \[
\pi: \mathcal C_{0,n(v)}(\sqrt[r]{L/D}, \beta(v))\rightarrow \overline{\mathcal{M}}_{0,n(v)}(\sqrt[r]{L/D}, \beta(v))
\]
 is the universal curve, 
\[
\mathcal L\rightarrow \mathcal C_{0,n(v)}(\sqrt[r]{L/D}, \beta(v))
\] 
is the universal $r$-th root. Here $n(v)$ is the number of marked points of the vertex $v$ and $\beta(v)$ is the degree assigned to the vertex $v$.
    
    \item $d_e$ is the degree of the edge $e\in E(v)$.
    \item $\text{ev}_e$ is the evaluation map at the node corresponding to $e$.
    \item $\bar{\psi}_{(e,v)}$ is the descendant class at the marked point corresponding to the pair $(e,v)$.
\end{itemize}

As the curve class is a fiber class of the fibration $Y\rightarrow D$ and the marking $p_i$ carries the contact order $k_i$ with $D_\infty$, the edge should carry contact order $k_i$ with $D_0$. Therefore, the contact order of the edge with the divisor $D\times \mathbb P^1$ is $k_i$. We can see from (\ref{eqn:r_dep}) that there is only negative powers of $t^\prime$. 
\end{proof}


In summary, there is one vertex $v_i$ for each marking $p_i$, $i\in \{1,\ldots,n\}$ and the curve classes are fiber classes of $P_Y\rightarrow D$. Each vertex $v_i$ carries the contact order $(k_i,k_i)$ with $Y_\infty+Y_\sigma$ at $p_i$. The unique edge associated with the vertex $v_i$ carries the contact order $k_i$ with $D\times \mathbb P^1$ and contact order zero with $Y_\infty$ and $Y_\sigma$ (This follows from Lemma \ref{lemma-edge} but it is even more special here, since there is only one edge. The unique edge cannot have non-zero contact with both $Y_\infty$ and $Y_\sigma$ at the same time and the total contact order with $Y_\infty$ and $Y_\sigma$ should be the same for $v_i$, so the unique edge has contact order zero with both $Y_\infty$ and $Y_\sigma$). The remaining part of the degeneration graph carries a total contact order $0$ with $X_\infty$, as well as with $X_\sigma$.

\subsubsection{Other vertices}\label{sec:other-vertices}

Now we examine other vertices over $P_Y$. For a vertex $v^\prime$ over $P_Y$ that is not $\{v_i\}_{i=0}^n$, recall that, by Lemma \ref{lemma-edge}, the edges have contact orders zero with $Y_\infty$ and $Y_\sigma$ and positive contact orders with $D\times \mathbb P^1$. Therefore, there is no fiber class along the fibration $P_Y\rightarrow Y$, the curves must be in one of the sections. Suppose that the curves are in $Y_\infty$. Then the virtual cycle becomes the virtual cycle of the moduli space of relative stable maps to $(Y_\infty,D_0+D_\infty)$. However, since we have positive contact along $D_0$, we also need a relative marking with positive contact with $D_\infty$ because $D$ is nef, but there are no such relative markings here. Hence, the curves cannot be in $Y_\infty$.


We have the following observation. 
\begin{lemma}\label{lemma-v-prime-insertion}
    The insertions (cohomology weights) of the edges associated with $v^\prime$ along the $\mathbb P^1$-direction are of the form $[\on{1}_{\mathbb P^1}]$.
\end{lemma} 
\begin{proof}
   Otherwise, if there is an edge whose insertion along the $\mathbb P^1$-direction is of the form $[\on{pt}_{\mathbb P^1}]$, then we can take the point $\on{pt}$ to be the point at $\infty\in \mathbb P^1$. Then the curve will be in $Y_\infty$, which cannot happen.  
\end{proof}


Therefore, the virtual cycle associated with $v^\prime$ is the virtual cycle for the local-relative Gromov--Witten theory of $(\mathcal O_{Y}(-D_\infty),D_0)$. As $D_\infty$ is nef, we can apply the local-relative correspondence \cite{vGGR}. This becomes the virtual cycle of the relative stable map to $(Y, D_0+D_\infty)$ with maximal contact along $D_\infty$. By Lemma \ref{lemma-rubber}, the virtual cycle vanishes after pushing forward, unless there are only two markings and the curve class is the fiber class. Therefore, for each of these vertices, there can be only one edge associated with this vertex.


We recall from Corollary \ref{cor-no-fiber} that, for vertices over $X\times \mathbb P^1$, as there is no fiber class along the $\mathbb P^1$-direction either, the virtual cycle for $(\mathbb P^1,0+\infty)$ becomes $\mathbb P^1\times \overline{M}_{0,n_v}$. Therefore, each vertex over $(X\times \mathbb P^1)$ has exactly one edge whose insertion (cohomology weight) along the $\mathbb P^1$-direction is of the form $[\on{pt}_{\mathbb P^1}]$, and other edges should carry insertions (cohomology weights) of the form  $[1_{\mathbb P^1}]$ along the $\mathbb P^1$-direction.

\begin{lemma}
The term of the degeneration graph $\Gamma$ in the degeneration formula vanishes unless the marking $p_0$ is distributed to a vertex on the $P_Y$-side.
\end{lemma}
\begin{proof}

Suppose $p_0$ is distributed to a vertex $v_0$ over $(X\times \mathbb P^1)$. As the evaluation with the class $[X_\infty-p^*D]=[X_0]$ means that $p_0$ maps to $X_0$, it gives a point constraint along the $\mathbb P^1$-direction under the degeneration.  Then the edges associated with $v_0$ should all have cohomology weights of the form $[\on{1}_{\mathbb P^1}]$ along the $\mathbb P^1$-direction. By the pairing, if a vertex $v$ over $P_Y$ is connected to $v_0$ via an edge $e$, then the cohomology weight of edge on the $v$-side is $[\on{pt}_{\mathbb P^1}]$ along the $\mathbb P^1$-direction.  Therefore, by Lemma \ref{lemma-v-prime-insertion}, these edges cannot connect to a vertex $v^\prime$ over $P_Y$ if $v^\prime$ is not one of $v_i$ that contains one of $p_i$, for $i=1,\ldots,n$. As other vertices over $P_Y$ each can only have one edge as well, to form a tree, there cannot be vertices over $P_Y$ other than $\{v_i\}_{i=1}^n$. By Lemma \ref{lemma-P-Y-vanish}, each $v_i$ has one edge with contact order $k_i$ to $D\times \mathbb P^1$ and the curve classes are fiber classes of $P_Y\rightarrow D$. This is a contradiction because the total contact order with $D\times \mathbb P^1$ is $D\cdot\beta>\sum_{i=1}^n k_{i}$, so there should be at least one more vertex over $P_Y$. 

\end{proof}

We conclude that $p_0$ has to be distributed to the $P_Y$-side. Similar to other vertices, the vertex $v_0$, containing $p_0$ and mapping to the $P_Y$-side also only has one edge with an insertion (cohomology weight)  of the form $[\on{1}_{\mathbb P^1}]$ along the $\mathbb P^1$-direction. 


\subsubsection{The degeneration graphs}

As there is only one edge for each vertex over $P_Y$, there can be only one vertex $v$ over $(X\times \mathbb P^1)$ and $v$ is connected to a vertex over $P_Y$ via an edge. This implies that the degeneration graph is of the star shape.

Now we know that, by Lemma \ref{lemma-v-prime-insertion}, each vertex over $P_Y$ that is not $\{v_i\}_{i=1}^n$ has one edge whose cohomology weight along the $\mathbb P^1$-direction is of the form $[\on{1}_{\mathbb P^1}]$. And, by Corollary \ref{cor-no-fiber}, each vertex over $(X\times \mathbb P^1)$ has exactly one edge whose cohomology weight along the $\mathbb P^1$-direction is of the form $[\on{pt}_{\mathbb P^1}]$. As there is only one vertex over $(X\times \mathbb P^1)$, there can be only one vertex over $P_Y$ other than $\{v_i\}_{i=1}^n$.

Therefore, there is exactly one vertex $v_0$ over $P_Y$ other than the vertices $\{v_i\}_{i=1}^n$ and there is one unique vertex $v$ over $X\times \mathbb P^1$. The marking $p_0$ is attached to $v_0$.  We conclude the following:

\begin{proposition}\label{prop-deg-graph}
The term of the degeneration graph $\Gamma$ in the degeneration formula vanishes unless    there is exactly one vertex over $(X\times \mathbb P^1)$ with $(n+1)$-edges connecting this vertex to the $(n+1)$-vertices over $P_Y$. We denote these vertices over $P_Y$ by $v_i$ for $i=0,\ldots,n$. Each $v_i$ contains one leg/marking $p_i$. 
\end{proposition}

\subsection{Conclusion}\label{sec:conclusion}

For $v_i$, $i=1,\ldots,n$, the orbifold marking $p_i$ with contact order $(k_i,k_i)$ is attached to this vertex $v_i$. There is one more marking $p_{e_i}$ which is coming from the unique edge $e_i$ attached to $v_i$. The $p_{e_i}$ has positive contact with $D\times \mathbb P^1$ and has a contact order zero to both $Y_{\infty}$ and $Y_\sigma$. At this marking, the insertion  along the $\mathbb P^1$-direction should be of the form $[\on{pt}_{\mathbb P^1}]$. By Lemma \ref{lemma-rubber}, the virtual cycle vanishes unless the curve class is a fiber class and the pushforward of the virtual cycle is the same as the virtual cycle for the moduli space of relative stable maps to $(Y,D_0+D_\infty)$ with fiber curve class and two markings. Recall that there is an automorphism factor $\eta_0^i=k_i$ in the degeneration formula. On the other hand, along the $\mathbb P^1$-direction, we have a relative invariant of $(\mathbb P^1,0+\infty)$ with maximal contacts to both divisor at two different markings and it is straightforward to compute that
\[
\langle [1]_{k_i}| \, |[1]_{k_i} \rangle_{0,2,k_i}^{(\mathbb P^1,0+\infty)}=\frac{1}{k_i}.
\]
This cancels with the automorphism factor $\eta_0^i=k_i$.
Therefore, the virtual cycle associated with $v_i$ times $\eta_0^i$ contributes $1$ in the degeneration formula. 

 For $v_0$, we have the interior marking $p_0$.  We recall that there is no fiber classes along the fibration $P_Y\rightarrow Y$ by Lemma \ref{lemma-edge}. 
The marking $p_0$ is mapped to $Y_0$, so the curve is entirely in $Y_0$. 
As the marking maps to $Y_0$, the virtual cycle just becomes the virtual cycle for the local-relative Gromov--Witten theory of $(\mathcal O_{Y}(-D_\infty),D_0)$:
\[
[\overline{M}_{\Gamma_2,v_0}(\mathcal O_{Y}(-D_\infty),D_0)]^{\on{vir}}\cap \on{ev}_0^*(-[D_\infty]).
\]
Recall that $D_\infty$ is nef. Then by the local-relative correspondence \cite{vGGR} (more precisely, we refer to a slightly different form stated in \cite{TY20b}*{Equation (2)}.), we have the virtual cycle for genus zero relative Gromov--Witten theory of $(Y,D_0+D_\infty)$:
\[
[\overline{M}_{\Gamma_2,v_0}(\mathcal O_{Y}(-D_\infty),D_0)]^{\on{vir}}\cap \on{ev}_0^*(-[D_\infty])=(-1)^{k_0}\pi_{\on{rel},*}[\overline{M}_{\Gamma_2,v_0}(Y,D_\infty+D_0)]^{\on{vir}}.
\]
By Lemma \ref{lemma-rubber}, the cycle vanishes after pushforward unless it is of degree zero in $D$ and has two markings.  The contact orders for both markings are $D\cdot \beta-\sum_{i=1}^n k_i=k_0$ to $D_0$ and $D_\infty$ respectively. The automorphism factor $\eta_0^0=k_0$ cancels again with the pushforward of the moduli space. Note that the sign $(-1)^{k_0}$ appears when we apply the local-relative correspondence. The total contribution from the vertex $v_0$ is $(-1)^{k_0}$.

For other interior markings, as we concluded above that each vertex $v_i$, for $i=0,\ldots,n$, can only have two markings. One of the markings is $v_i$ and the other marking is the relative marking corresponding to the edge. Therefore, these interior markings are distributed to the vertex over the $X\times \mathbb P^1$-side.

Combining the above, the virtual cycle for the vertex over the $X\times \mathbb P^1$-side is the virtual cycle on the LHS of (\ref{iden-main}). This concludes the proof.

\begin{remark}
    If we remove the vertices $v_1,\ldots,v_n$ and their associated edges, the degeneration graph has one vertex over $X\times \mathbb P^1$, one vertex over $P_Y$ and one edge connecting these two vertices. There are no fiber classes remaining. Indeed, this remaining part of the degeneration graph is the same as the degeneration graph in the proof of \cite{vGGR}*{Theorem 1.1}. Therefore, this gives the one extra relative marking for $(X,D)$.  
\end{remark}


\section{Special cases}

\subsection{The local-relative correspondence}

Recall that we can rewrite the local-relative correspondence of \cite{vGGR} as follows:

\begin{align}\label{iden-local-rel}
\langle [\iota^*\gamma]_{D\cdot \beta}\rangle_{0,1,\beta}^{(X,D)}=(-1)^{D\cdot \beta+1}\langle [\gamma\cup D]\rangle_{0,1,\beta}^{\mathcal O_X(-D)},
\end{align}
where $D\cdot\beta>0$.

Recall that we consider a section $\sigma$ of $\mathcal O_X(D)$ such that the zero locus of $\sigma$ is the divisor $D$. The section $\sigma$ induces a section $\tilde f$ of $\mathcal O_{P}(1)$ such that the zero locus $\tilde f^{-1}(0)$ of $\tilde f$ is $X_\sigma:=\sigma(X)\cong X$ and the intersection of $\tilde f^{-1}(0)$ and $X_{\infty}$ is $D$. See, for example, \cite{FTY}*{Section 3.1}.

On the other hand, for the special case of Theorem \ref{thm-main-1} when $n=n_0=0$, we have

\begin{align*}
\langle [\iota^*\gamma]_{D\cdot \beta}\rangle_{0,1,\beta}^{(X,D)}=(-1)^{D\cdot \beta}\langle [\gamma\cup (X_\infty-\pi^*D)]_0\rangle_{0,1,\beta+0f}^{(P,X_\infty+X_\sigma)}.
\end{align*}
 Since the curve class is $\beta+0f$, the curves have no fiber components. As the marking maps to $X_0$, we actually have the local invariant:
\begin{align*}
   (-1)^{D\cdot \beta}\langle [\gamma\cup (X_\infty-\pi^*D)]_0\rangle_{0,1,\beta+0f}^{(P,X_\infty+X_\sigma)}&=(-1)^{D\cdot \beta}\langle [\gamma\cup (-D)]_0\rangle_{0,1,\beta}^{\mathcal O_X(-D)}\\
   &=(-1)^{D\cdot \beta+1}\langle [\gamma\cup D]_0\rangle_{0,1,\beta}^{\mathcal O_X(-D)}. 
\end{align*}
This is precisely the local-relative correspondence (\ref{iden-local-rel}).


\subsection{Relative invariants with two contacts}

In this section, we specialize Theorem \ref{thm-main-1} to relative invariants with two contacts. Theorem \ref{thm-main-1} becomes

\begin{equation}
          \begin{split}
          &\pi_{\on{rel},*}\left([\overline{M}_{0,(k_0,k_1),\beta}(X,D)]^{\on{vir}}\right)      \\
 =      &   (-1)^{k_0} u_{*}\left( [\overline{M}_{0,((0,0),(k_1,k_1)),\beta+k_1f}(P,X_{\infty}+ X_{\sigma})]^{\on{vir}}\cap\on{ev}_0^*(X_\infty-\pi^*D)\right).
     \end{split}
        \end{equation}   
The last line features orbifold invariants of  $(P, X_{\infty}+X_\sigma)$ with maximal contact, so we can apply the local-orbifold correspondence of \cite{BNTY}.

\begin{corollary}\label{cor-2-pt}
  Given a smooth projective variety $X$ with a smooth nef divisor $D$, genus zero relative invariants of $(X,D)$ satisfy the following identity:
        \begin{equation}
          \begin{split}
          & \pi_{\on{rel},*}\left([\overline{M}_{0,(k_0,k_1),\beta}(X,D)]^{\on{vir}}\right)\\    
 =         & (-1)^{k_0} u_{*}\left(\on{ev}_1^*(X_{\infty}\cup X_{\infty})\cap [\overline{M}_{0,2,\beta+k_1f}(\mathcal O_{P}(-X_{\infty})\oplus \mathcal O_{P}( -X_{\infty}))]^{\on{vir}}\right)\in A_*(\overline{M}_{0,2,\beta}(X)),
 \end{split}
        \end{equation}
        where
the RHS is the virtual cycle for the genus zero local Gromov--Witten theory of $\mathcal O_{P}(-X_{\infty})\oplus \mathcal O_{P}( -X_{\infty})$.

\end{corollary}

Let $\gamma_0,\gamma_1\in H^*(X)$ and $k_0, k_1\in \mathbb Z_{>0}$, such that $k_0+ k_1=D\cdot \beta$, then we can write Corollary \ref{cor-2-pt} on the level of invariants
   \begin{equation} 
   \begin{split}
    \langle [\iota^*\gamma_0]_{k_0},[\iota^*\gamma_1]_{k_1}\rangle_{0,2,\beta}^{(X,D)}=(-1)^{k_0}\langle [\gamma_0\cup (X_\infty-\pi^*D)]_0, [\gamma_{1}\cup X_{\infty}\cup X_{\infty}]_{0}\rangle_{0,2,\beta+k_1f}^{\mathcal O_{P}(-X_{\infty})\oplus \mathcal O_{P}( -X_{\infty})},
    \end{split}
    \end{equation}
      where $\gamma_0$ and $\gamma_1$ are considered as classes in $H^*(P)$ via the pullback induced by the projection $\pi: P\rightarrow X$.


\section{Computing relative Gromov--Witten invariants}\label{sec: computation}
In this section, we use the special case of Theorem \ref{thm-main-1} in Corollary \ref{cor-2-pt} to explain how it simplifies the computation of relative Gromov--Witten invariants.

In \cite{You22}*{Definition 1.1} (see also \cite{GRZ} and \cite{Wang}), the proper Landau--Ginzburg potential  of a smooth log Calabi--Yau pair $(X,D)$ is written as a generating function of two-point relative Gromov--Witten invariants. When $D$ is nef, this generating function of two-point invariants is computed in \cite{You22} via the relative mirror theorem of \cite{FTY}. One can see from \cite{You22}*{Section 4 and Section 5} that the computation of two-point invariants is quite complicated. By Corollary \ref{cor-2-pt}, these invariants are local invariants of a rank $2$ bundle over $P$. These local invariants can be computed via mirror symmetry. 

\subsection{Identifying the invariants}

We further specialize to the case of two-point relative invariants of a smooth log Calabi--Yau pair $(X,D)$ with  insertions $[\on{pt}]_n$ and $[1]_1$:
\[
\langle [\on{pt}]_{n}, [1]_{1}\rangle_{0,2,\beta}^{(X,D)},
\]
where $D\cdot \beta=n+1$.  After multiplying these invariants by $n$, we have the terms in the theta function $\vartheta_1$ in \cite{GRZ}, \cite{You22} and \cite{Wang}:
\[
\vartheta_1:=x^{-1}+x^{-1}\sum_{n=1}^\infty\sum_{\beta:D\cdot\beta=n+1} n\langle [\on{pt}]_{n}, [1]_{1}\rangle_{0,2,\beta}^{(X,D)}q^\beta.
\]

We first need to prove the following identity which was proved in \cite{GRZ}*{Theorem 5.4} in dimension $2$ which is a generalization of \cite{CC}*{Theorem 6.6}.
\begin{lemma}
    For a smooth log Calabi--Yau pair $(X,D)$, the following identity holds.
    \[
n\langle [\on{pt}]_{n}, [1]_{1}\rangle_{0,2,\beta}^{(X,D)}=\frac{1}{n}\langle [1]_{n},[\on{pt}]_{1}\rangle_{0,2,\beta}^{(X,D)}.
\]
\end{lemma}
\begin{proof}
    This is a result of the WDVV equation for the relative Gromov--Witten theory of $(X,D)$ \cite{FWY}*{Proposition 7.5}. The proof is similar to that of \cite{GRZ}*{Theorem 5.4}. Set
    \[
    N_{k,p}^\beta:=\langle [\on{pt}]_{k},[1]_{p}\rangle_{0,2,\beta}^{(X,D)}.
    \]
    By \cite{You22}*{Proposition 3.7 \& Remark 3.8}, the WDVV equation implies the following identity:
    \begin{equation}\label{iden-wdvv}
    \begin{split}
&(k+1)N_{k+1,p}^\beta+(k+p)N_{k+p,1}^\beta+\sum_{a,b>0, a+b=k}\sum_{\beta_1+\beta_2=\beta}abN_{a,1}^{\beta_1} N_{b,p}^{\beta_2}\\
=& kN_{k,p+1}^\beta+ \sum_{r=1}^{p-1}(p-r)\sum_{\beta_1+\beta_2=\beta}N_{p-r,1}^{\beta_1}kN_{k,r}^{\beta_2},
\end{split}
\end{equation}
where $D\cdot \beta=k+p+1$. Set $n:=k+p$, we can rewrite (\ref{iden-wdvv}) as
\begin{equation*}
    \begin{split}
        &kN_{k,n-k+1}\\
        =&(k+1)N_{k+1,n-k}^\beta+nN_{n,1}^\beta+\sum_{a,b>0, a+b=k}\sum_{\beta_1+\beta_2=\beta}abN_{a,1}^{\beta_1} N_{b,n-k}^{\beta_2}-\sum_{r=1}^{n-k-1}(n-k-r)\sum_{\beta_1+\beta_2=\beta}N_{n-k-r,1}^{\beta_1}kN_{k,r}^{\beta_2}.
    \end{split}
\end{equation*}

Following the proof of \cite{GRZ}*{Theorem 5.4}, we specialize this equation to $k=n,,n-1,1$, we have
\[
nN_{n,1}=nN_{n,1},
\]
\[
(n-1)N_{n-1,2}=nN_{n,1}+nN_{n,1}+\ldots,
\]
\[
\vdots
\]
\[
N_{1,n}=2N_{2,n-1}+nN_{n,1}+\ldots.
\]
Summing over these $n$ identities and subtracting $\sum_{k=2}^n k N_{k,n-k+1}$ to both sides, we have
\[
N_{1,n}=n^2N_{n,1}+\sum_{k=1}^n\sum_{a,b>0, a+b=k}\sum_{\beta_1+\beta_2=\beta}abN_{a,1}^{\beta_1} N_{b,n-k}^{\beta_2}-\sum_{k=1}^n\sum_{r=1}^{n-k-1}(n-k-r)\sum_{\beta_1+\beta_2=\beta}N_{n-k-r,1}^{\beta_1}kN_{k,r}^{\beta_2}.
\]
Observe that the two triple sums cancel with each other. Therefore, we have
\[
N_{1,n}=n^2N_{n,1}.
\]
We conclude that
\[
n\langle [\on{pt}]_{n}, [1]_{1}\rangle_{0,2,\beta}^{(X,D)}=\frac{1}{n}\langle [1]_{n},[\on{pt}]_{1}\rangle_{0,2,\beta}^{(X,D)}.
\]
\end{proof}

We assume that $D$ is nef, then we have
\begin{align*}
    n\langle [\on{pt}]_{n}, [1]_{1}\rangle_{0,2,\beta}^{(X,D)}&=\frac{1}{n}\langle [1]_{n},[\on{pt}]_{1}\rangle_{0,2,\beta}^{(X,D)}\\
    &=(-1)^{n}\frac{1}{n}\langle [X_{\infty}-\pi^*D],[\on{pt}]\rangle_{0,2,\beta+f}^{\mathcal O_{P}(-X_{\infty})\oplus \mathcal O_{P}( -X_{\infty})}\\
    &=(-1)^{n+1}\langle [\on{pt}]\rangle_{0,1,\beta+f}^{\mathcal O_{P}(-X_{\infty})\oplus \mathcal O_{P}( -X_{\infty})}.
\end{align*}

The computation of these two-point relative invariants are now reduced to the computation of one-point local invariants of $\mathcal O_{P}(-X_{\infty})\oplus \mathcal O_{P}( -X_{\infty})$.

\subsection{Mirror theorem for toric bundles}

The bundle 
\[
E:=\mathcal O_{P}(-X_{\infty})\oplus \mathcal O_{P}( -X_{\infty})
\]
is a toric bundle over $X$. By the mirror theorem for toric bundles \cite{Brown}, $1$-point genus zero invariants of $\mathcal O_{P}(-X_{\infty})\oplus \mathcal O_{P}( -X_{\infty})$ can be computed in terms of genus zero invariants of $X$ and combinatorial data. 

\begin{remark}
    As the mirror theorem of \cite{Brown} is stated for compact targets, which is not the case here. To be more precise, we can either use the mirror theorem of \cite{JTY} for toric stack bundles (a generalization of \cite{CCIT15}) or consider the twisted $I$-function of $P$ by the bundle $E$.
\end{remark}

Recall that the small $J$-function for absolute Gromov--Witten theory of $E$ is
\[
J_{E}(q,z)=e^{\sum_{i=0}^r p_i\log q_i/z}\left(z+\sum_{ \beta\in \on{NE(X)}\setminus\{0\},n\geq 0}\sum_{\alpha}q^{\beta}q_0^n\left\langle \frac{\phi_{P,\alpha}}{z-\psi}\right\rangle_{0,1, \beta+nf}^{E}\phi_{P}^{\alpha}\right),
\]
where $\{p_i\}_{i=1}^r$ is an integer, nef basis of $H^2(X)$; $\{p_i\}_{i=0}^r$ is an integer, nef basis of $H^2(P)$; $\tau_{0,2}=\sum_{i=0}^r p_i \log q_i\in H^2(P)$; $\{\phi_{P,\alpha}\}$ is a basis of $H^*(P)$; $\{\phi_{P}^\alpha\}$ is the dual basis under the Poincar\'e pairing.

Now, we consider the $I$-function of $E$. We choose

\[
p_0:=h:=c_1(\mathcal O_{P}(1)).
\]
The $I$-function of $E$ is
\begin{align}
I_{E}(y,z)= \sum_{\beta\in \on{NE}(X),n\geq 0} J_{X, \beta}(q,z)y^{\beta}y_0^n\frac{\prod_{a\leq 0}(h+az)}{\prod_{a\leq n}(h+az)}\frac{\prod_{a\leq 0}(h-D+az)}{\prod_{a\leq n-D\cdot \beta}(h-D+az)}\frac{\prod_{a\leq 0}(-h+az)^2}{\prod_{a\leq -n}(-h+az)^2},
\end{align}
where $J_{X, \beta}(q,z)$ is the degree $\beta$-part of the $J$-function of $X$. That is,
\begin{align*}
J_{X}(q,z)& =\sum_{\beta} J_{X, \beta}(q,z)q^\beta\\
&=e^{\sum_{i=1}^r p_i\log q_i/z}\left(z+\sum_{ \beta\in \on{NE(X)}\setminus\{0\}}\sum_{\alpha}q^{\beta}\left\langle \frac{\phi_{X,\alpha}}{z-\psi}\right\rangle_{0,1, \beta}^{X}\phi_X^{\alpha}\right).
\end{align*}


We assume that $(X,D)$ is a log Calabi--Yau pair. Therefore, $E$ is local Calabi--Yau. The mirror theorem of \cite{Brown} implies that
\[
J_E(q(y),z)=I_E(y,z),
\]
where the mirror map is given by the coefficient of $z^0$. We compute that
\[
\sum_{i=0}^{r} p_i\log q_i=\sum_{i=1}^r p_i\log y_i+p_0\log y_0+\sum_{\substack{\beta\in \on{NE}(X)\\ D\cdot\beta\geq 2}} \langle [\on{pt}]\psi^{D\cdot \beta-2}\rangle_{0,1,\beta}^Xy^\beta(-1)^{D\cdot\beta-1}(D\cdot\beta-1)!(h-D).
\]


Therefore, the mirror map is given by
\[
\sum_{i=1}^{r} p_i\log q_i=\sum_{i=1}^r p_i\log y_i-g(-y)D,
\]
and
\[
p_0\log q_0=p_0\log y_0+g(-y)p_0,
\]
where
\[
g(-y)=\sum_{\substack{\beta\in \on{NE}(X)\\ D\cdot\beta\geq 2}} \langle [\on{pt}]\psi^{D\cdot \beta-2}\rangle_{0,1,\beta}^Xy^\beta(-1)^{D\cdot\beta-1}(D\cdot\beta-1)!.
\]
Therefore, 
\begin{align}\label{mirror-map}
q^\beta=\exp(g(-y)D\cdot\beta)y^\beta
\end{align}
and
\begin{align}\label{mirror-map-0}
q_0=y_0\exp\left(-g(-y)\right).
\end{align}

\subsection{Computation via the mirror theorem}
We would like to extract the coefficient of $\frac{1}{z}$ that takes value in $h^2$. The coefficient for the $J$-function is
\[
\sum_{ \beta+nf\in \on{NE(E)}\setminus\{0\},n\geq 0} q^\beta q_0^n\langle [\on{pt}]\rangle_{0,1,\beta+nf}^E,
\]
where $[\on{pt}]\in H^*(X)$ is the point-class of $H^*(X)$. We further take the coefficient of $q_0$,
then we have
\[
q_0+\sum_{ \beta\in \on{NE(X)}\setminus\{0\}} q^\beta q_0\langle [\on{pt}]\rangle_{0,1,\beta+f}^E,
\]
where the term $q_0$ comes from the contribution from $\beta=0$.

Now we consider the corresponding coefficient of the $I$-function. By the change of the variables (\ref{mirror-map}) and (\ref{mirror-map-0}), the coefficient of $q_0$ of the $J$-function is exactly the coefficient of $y_0$ of the $I$-function under the change of variables $q_0=y_0\exp\left(-g(-y)\right)$. Therefore, we will extract the coefficient of $\frac{y_0}{z}h^2$ of the $I$-function.

For $n=1$, that is, the coefficient of $y_0^1$, we have 
\[
\frac{\prod_{a\leq 0}(h+az)}{\prod_{a\leq 1}(h+az)}=\frac{1}{h+z},
\]
and 
\[
\frac{\prod_{a\leq 0}(-h+az)^2}{\prod_{a\leq -1}(-h+az)^2}=h^2.
\]
Therefore, to extract the coefficient of $\frac{y_0}{z}h^2$, we just need to extract the coefficient of $\frac{1}{z}$ that takes value in $[1]\in H^*(X)$ of
\[
\sum_{\beta\in \on{NE}(X)} J_{X, \beta}(q,z)y^{\beta}\frac{1}{h+z}\frac{\prod_{a\leq 0}(h-D+az)}{\prod_{a\leq 1-D\cdot \beta}(h-D+az)}.
\]
For $\beta=0$, the coefficient is $1$. For $\beta\neq 0$, 
\[
J_{X, \beta}(q,z)=\sum_{ \beta\in \on{NE(X)}\setminus\{0\}}\sum_{\alpha}q^{\beta}\left\langle \frac{\phi_{X,\alpha}}{z-\psi}\right\rangle_{0,1, \beta}^{X}\phi_X^{\alpha}.
\]
For $\phi_X^{\alpha}=[1]$, we have $D\cdot \beta\geq 2$, and
\[
\left\langle \frac{\phi_{X,\alpha}}{z-\psi}\right\rangle_{0,1, \beta}^{X}=\langle  [\on{pt}]\psi^{D\cdot\beta-2}\rangle_{0,1,\beta}^{X}\frac{1}{z^{D\cdot\beta-1}}.
\]
We also have
\[
\frac{\prod_{a\leq 0}(h-D+az)}{\prod_{a\leq 1-D\cdot \beta}(h-D+az)}=(-1)^{D\cdot\beta-1}\prod_{0\leq a\leq D\cdot \beta-2}(D-h+az)
\]
and
\[
\frac{1}{h+z}=\frac{1}{z}\sum_{k\geq 0} \frac{(-h)^k}{z^k}.
\]
This shows that there is no coefficient of $\frac{1}{z}$ that takes value in $[1]\in H^*(X)$ when $\beta\neq 0$. Therefore, the coefficient here is just $1$. 


By the mirror theorem, we have
\[
q_0+\sum_{ \beta\in \on{NE(X)}\setminus\{0\}} q^\beta q_0\langle [\on{pt}]\rangle_{0,1,\beta+f}^E=y_0.
\]
Therefore,
\begin{equation}
    \begin{split}
        &1+\sum_{ \beta\in \on{NE(X)}\setminus\{0\}} q^\beta \langle [\on{pt}]\rangle_{0,1,\beta+f}^E\\
        =&y_0/q_0\\
        =&\exp\left(-g(-y)\right)\\
        =&\exp\left(\sum_{\substack{\beta\in \on{NE}(X)\\ D\cdot\beta\geq 2}} \langle [\on{pt}]\psi^{D\cdot \beta-2}\rangle_{0,1,\beta}^Xy^\beta(-1)^{D\cdot\beta}(D\cdot\beta-1)!\right).
    \end{split}
\end{equation}

After a change of variable $y\mapsto -y$, we have
\begin{align*}
&x^{-1}\vartheta_1\\
=&1+\sum_{n}\sum_{\beta: D\cdot\beta=n+1} n\langle [\on{pt}]_{1}, [1]_{n}\rangle_{0,\beta,2}^{(X,D)}q^\beta\\
=&1+\sum_{n}\sum_{\beta: D\cdot\beta=n+1}(-1)^{n+1}\langle [\on{pt}]\rangle_{0,1,\beta+f}^{\mathcal O_{P_1}(-X_{\infty,1})\oplus \mathcal O_{P_1}( -X_{\infty,1})}q^\beta\\
=& \exp\left(g(y(q))\right),
\end{align*}
where $y(q)$ is the inverse of the mirror map.
We recover the computation of the theta function $\vartheta_1$ in \cite{You22}*{Theorem 5.1} (=Theorem \ref{thm-lg-potential}.)

\section{Proof of Theorem \ref{thm-local-orb}}\label{sec:proof-local-orb}

 The proof of Theorem \ref{thm-local-orb} is almost identical to the proof of Theorem \ref{thm-main-1}. We now explain why the proof is still true when additional root constructions are added.

The smooth pair $(X,D)$ in Section \ref{sec:proof} is replaced by the orbifold pair $(\mathcal X_{m-1},\mathcal D_m)=(P_m,\pi^*_m D_1+\cdots+\pi^*_m D_{m-1}+X_{\infty,m}+X_{\sigma,m})$.

 \begin{itemize}
     \item[{\bf Step 1}] We consider the degeneration of $P_m: =\mathbb P(\mathcal O_{X}(-D_m)\oplus \mathcal O_X)$ as in Section \ref{sec:deg}. Then we degenerate the pair $(\mathcal P_m,\mathcal X_{\infty,m}+\mathcal X_{\sigma,m})$ into
\[
(\mathcal X_{m-1}\times \mathbb P^1, (\mathcal D_{m}\times \mathbb P^1)\cup (\mathcal X_{m-1}\times \infty)\cup (\mathcal X_{m-1}\times \on{pt})) \text{ and } (\mathcal P_{Y_m}, (\mathcal D_m\times \mathbb P^1)\cup \mathcal Y_{\infty,m}\cup \mathcal Y_{\sigma,m}),
\]
where the notation is similar to Section \ref{sec:proof}:
\begin{itemize}
\item $Y_m:=\mathbb P(N_{D_m/X}\oplus \mathcal O_{D_m})=\mathbb P(\mathcal O_{D_m}(-D_m)\oplus \mathcal O_{D_m})$. 
    \item $D_{\infty,m}$ is the $\infty$-section of $Y_m$.
    \item $P_{Y_m}:=\mathbb P(\mathcal O_Y(-D_{\infty,m})\oplus \mathcal O)$.
    \item $Y_{\infty,m}$ and $Y_{\sigma,m}$ are infinity and $\sigma$-divisors of $P_{Y_m}$ as before.
    \item $\mathcal P$, $\mathcal D$, $\mathcal X$, $\mathcal Y$'s are obtained from $P$, $D$, $X$, $Y$ by adding root constructions along (the pullback of) $D_1+\cdots+D_{m-1}$ respectively.
\end{itemize}
\item[{\bf Step 2}] The degeneration formula of \cite{AF}*{Proposition 5.9.1} can still be applied for this degeneration. We just need to have additional labeling for ages along $D_i$ for $i=1,\ldots,m-1$ to record the additional orbifold structure. Lemma \ref{lemma-mid-age} still holds when we add root structures to $D_i$. So, no edges with mid-ages with respect to $X\times \infty, X\times \on{pt}, Y_\infty$ or  $Y_\sigma$ will appear in the degeneration formula.

\item[{\bf Step 3}] For virtual cycles on the $(X\times \mathbb P^1)$-side. The vanishing result of Lemma \ref{lemma-X-P-vanish} is along the $\mathbb P^1$-factor, so it is not affected by additional root constructions along $D_i$. On the $P_Y$-side, Lemma \ref{lemma-FTY-orb} is already stated in the orbifold setting. The vertex contribution in the localization formula used in Lemma \ref{lemma-rubber} is given in \cite{TY20c}*{Section 2.2.2} for orbifolds.

\item[{\bf Step 4}] Edges of the degeneration graph. Lemma \ref{lemma-edge} is still true, because the total contact orders with $X\times \infty$ and $X\times \on{pt}$ are still the same. Similarly for $Y_\infty$ and $Y_\sigma$. Again, there is no fiber class of the trivial fibration $\mathcal X_{m-1}\times \mathbb P^1\rightarrow \mathcal X_{m-1}$ and Corollary \ref{cor-no-fiber} is true.

\item[{\bf Step 5}] Vertices with $p_i$ attached.  We can still apply the orbifold quantum Lefschetz principle of \cite{Tseng}*{proposition 5.2.3} because the divisor $\mathcal D_m$ is pulled back from $D_m\subset X$, so $\mathcal O_{\mathcal X_{m-1}}(\mathcal D_m)$ and $\mathcal O_{\mathcal P_{Y,m}}(Y_\infty)$ are still convex. Therefore, Lemma \ref{lemma-p-i-p-0} is true. Lemma \ref{lemma-P-Y-vanish} uses the virtual localization computation in \cite{TY20c}*{Section 2.2.2} which was already stated for general orbifolds. So the conclusion of Section \ref{sec:vertices-p-i} remains unchanged.

\item[{\bf Step 6}] The analysis of other vertices is the same as Section \ref{sec:other-vertices}, except that we will need to use the orbifold version of the local-relative correspondence \cite{vGGR} which was proved in \cite{BNTY}*{Section 1.3}.

\item[{\bf Step 7}] From the previous steps, we can draw the same conclusion for the degeneration graph as Proposition \ref{prop-deg-graph}. Then the rest follows.

 \end{itemize}

\bibliographystyle{amsxport}
\bibliography{main}

\begin{bibdiv}
\begin{biblist}

\bib{AC}{article}{
      author={Abramovich, Dan},
      author={Chen, Qile},
       title={Stable logarithmic maps to {D}eligne-{F}altings pairs {II}},
        date={2014},
        ISSN={1093-6106},
     journal={Asian J. Math.},
      volume={18},
      number={3},
       pages={465\ndash 488},
         url={https://doi-org.ezproxy.uio.no/10.4310/AJM.2014.v18.n3.a5},
      review={\MR{3257836}},
}

\bib{ACGS}{article}{
      author={Abramovich, Dan},
      author={Chen, Qile},
      author={Gross, Mark},
      author={Siebert, Bernd},
       title={Punctured logarithmic maps},
        date={2020},
     journal={arXiv preprint arXiv:2009.07720},
}

\bib{ACW}{article}{
      author={Abramovich, Dan},
      author={Cadman, Charles},
      author={Wise, Jonathan},
       title={Relative and orbifold {G}romov-{W}itten invariants},
        date={2017},
     journal={Algebr. Geom.},
      volume={4},
      number={4},
       pages={472\ndash 500},
}

\bib{AF}{article}{
      author={Abramovich, Dan},
      author={Fantechi, Barbara},
       title={Orbifold techniques in degeneration formulas},
        date={2016},
        ISSN={0391-173X},
     journal={Ann. Sc. Norm. Super. Pisa Cl. Sci. (5)},
      volume={16},
      number={2},
       pages={519\ndash 579},
      review={\MR{3559610}},
}

\bib{AJT16}{article}{
      author={Andreini, Elena},
      author={Jiang, Yunfeng},
      author={Tseng, Hsian-Hua},
       title={Gromov-{W}itten theory of product stacks},
        date={2016},
        ISSN={1019-8385,1944-9992},
     journal={Comm. Anal. Geom.},
      volume={24},
      number={2},
       pages={223\ndash 277},
         url={https://doi.org/10.4310/CAG.2016.v24.n2.a1},
      review={\MR{3514559}},
}

\bib{AMW}{article}{
      author={Abramovich, Dan},
      author={Marcus, Steffen},
      author={Wise, Jonathan},
       title={Comparison theorems for {G}romov-{W}itten invariants of smooth
  pairs and of degenerations},
        date={2014},
        ISSN={0373-0956,1777-5310},
     journal={Ann. Inst. Fourier (Grenoble)},
      volume={64},
      number={4},
       pages={1611\ndash 1667},
         url={https://doi.org/10.5802/aif.2892},
      review={\MR{3329675}},
}

\bib{BFGW}{article}{
      author={Bousseau, Pierrick},
      author={Fan, Honglu},
      author={Guo, Shuai},
      author={Wu, Longting},
       title={Holomorphic anomaly equation for {$(\mathbb P^2,E)$} and the
  {N}ekrasov-{S}hatashvili limit of local {$\mathbb P^2$}},
        date={2021},
        ISSN={2050-5086},
     journal={Forum Math. Pi},
      volume={9},
       pages={Paper No. e3, 57},
         url={https://doi.org/10.1017/fmp.2021.3},
      review={\MR{4254953}},
}

\bib{BNR22}{article}{
      author={Battistella, Luca},
      author={Nabijou, Navid},
      author={Ranganathan, Dhruv},
       title={Gromov-{W}itten theory via roots and logarithms},
        date={2024},
        ISSN={1465-3060,1364-0380},
     journal={Geom. Topol.},
      volume={28},
      number={7},
       pages={3309\ndash 3355},
         url={https://doi.org/10.2140/gt.2024.28.3309},
      review={\MR{4833682}},
}

\bib{BNR24}{article}{
      author={Battistella, Luca},
      author={Nabijou, Navid},
      author={Ranganathan, Dhruv},
       title={Logarithmic negative tangency and root stacks},
        date={2024},
     journal={arXiv preprint arXiv:2402.08014},
}

\bib{BNTY}{article}{
      author={Battistella, Luca},
      author={Nabijou, Navid},
      author={Tseng, Hsian-Hua},
      author={You, Fenglong},
       title={The local-orbifold correspondence for simple normal crossing
  pairs},
        date={2023},
        ISSN={1474-7480,1475-3030},
     journal={J. Inst. Math. Jussieu},
      volume={22},
      number={5},
       pages={2515\ndash 2531},
         url={https://doi.org/10.1017/S1474748022000172},
      review={\MR{4624970}},
}

\bib{Brown}{article}{
      author={Brown, Jeff},
       title={Gromov-{W}itten invariants of toric fibrations},
        date={2014},
        ISSN={1073-7928,1687-0247},
     journal={Int. Math. Res. Not. IMRN},
      number={19},
       pages={5437\ndash 5482},
         url={https://doi.org/10.1093/imrn/rnt030},
      review={\MR{3267376}},
}

\bib{CC}{article}{
      author={Cadman, Charles},
      author={Chen, Linda},
       title={Enumeration of rational plane curves tangent to a smooth cubic},
        date={2008},
        ISSN={0001-8708,1090-2082},
     journal={Adv. Math.},
      volume={219},
      number={1},
       pages={316\ndash 343},
         url={https://doi.org/10.1016/j.aim.2008.04.013},
      review={\MR{2435425}},
}

\bib{CCIT15}{article}{
      author={Coates, Tom},
      author={Corti, Alessio},
      author={Iritani, Hiroshi},
      author={Tseng, Hsian-Hua},
       title={A mirror theorem for toric stacks},
        date={2015},
        ISSN={0010-437X},
     journal={Compos. Math.},
      volume={151},
      number={10},
       pages={1878\ndash 1912},
         url={https://doi-org.ezproxy.uio.no/10.1112/S0010437X15007356},
      review={\MR{3414388}},
}

\bib{CCLT}{article}{
      author={Chan, Kwokwai},
      author={Cho, Cheol-Hyun},
      author={Lau, Siu-Cheong},
      author={Tseng, Hsian-Hua},
       title={Gross fibrations, {SYZ} mirror symmetry, and open
  {G}romov-{W}itten invariants for toric {C}alabi-{Y}au orbifolds},
        date={2016},
        ISSN={0022-040X},
     journal={J. Differential Geom.},
      volume={103},
      number={2},
       pages={207\ndash 288},
         url={http://projecteuclid.org.ezproxy.uio.no/euclid.jdg/1463404118},
      review={\MR{3504949}},
}

\bib{CG}{article}{
      author={Coates, Tom},
      author={Givental, Alexander},
       title={Quantum {R}iemann-{R}och, {L}efschetz and {S}erre},
        date={2007},
        ISSN={0003-486X},
     journal={Ann. of Math. (2)},
      volume={165},
      number={1},
       pages={15\ndash 53},
  url={https://doi-org.proxy.lib.ohio-state.edu/10.4007/annals.2007.165.15},
}

\bib{CGT}{article}{
      author={Coates, Tom},
      author={Givental, Alexander},
      author={Tseng, Hsian-Hua},
       title={Virasoro constraints for toric bundles},
        date={2024},
     journal={Forum of Mathematics, Pi},
      volume={12},
       pages={e4},
}

\bib{chan}{article}{
      author={Chan, Kwokwai},
       title={A formula equating open and closed {G}romov-{W}itten invariants
  and its applications to mirror symmetry},
        date={2011},
        ISSN={0030-8730,1945-5844},
     journal={Pacific J. Math.},
      volume={254},
      number={2},
       pages={275\ndash 293},
         url={https://doi.org/10.2140/pjm.2011.254.275},
      review={\MR{2900016}},
}

\bib{Chen}{article}{
      author={Chen, Qile},
       title={Stable logarithmic maps to {D}eligne-{F}altings pairs {I}},
        date={2014},
        ISSN={0003-486X},
     journal={Ann. of Math. (2)},
      volume={180},
      number={2},
       pages={455\ndash 521},
         url={https://doi-org.ezproxy.uio.no/10.4007/annals.2014.180.2.2},
      review={\MR{3224717}},
}

\bib{FTY}{article}{
      author={Fan, Honglu},
      author={Tseng, Hsian-Hua},
      author={You, Fenglong},
       title={Mirror theorems for root stacks and relative pairs},
        date={2019},
        ISSN={1022-1824},
     journal={Selecta Math. (N.S.)},
      volume={25},
      number={4},
       pages={Art. 54, 25},
  url={https://doi-org.proxy.lib.ohio-state.edu/10.1007/s00029-019-0501-z},
      review={\MR{3997137}},
}

\bib{FW}{article}{
      author={Fan, Honglu},
      author={Wu, Longting},
       title={Witten-{D}ijkgraaf-{V}erlinde-{V}erlinde equation and its
  application to relative {G}romov-{W}itten theory},
        date={2021},
        ISSN={1073-7928,1687-0247},
     journal={Int. Math. Res. Not. IMRN},
      number={13},
       pages={9834\ndash 9852},
         url={https://doi.org/10.1093/imrn/rnz353},
      review={\MR{4283567}},
}

\bib{FWY}{article}{
      author={{Fan}, Honglu.},
      author={{Wu}, Longting.},
      author={{You}, Fenglong.},
       title={Structures in genus-zero relative {G}romov--{W}itten theory},
        date={2020},
     journal={J. Topol.},
      volume={13},
      number={1},
       pages={269\ndash 307},
}

\bib{Gathmann03}{article}{
      author={Gathmann, Andreas},
       title={Relative {G}romov-{W}itten invariants and the mirror formula},
        date={2003},
        ISSN={0025-5831,1432-1807},
     journal={Math. Ann.},
      volume={325},
      number={2},
       pages={393\ndash 412},
         url={https://doi.org/10.1007/s00208-002-0345-1},
      review={\MR{1962055}},
}

\bib{GRZ}{article}{
      author={Gr\"afnitz, Tim},
      author={Ruddat, Helge},
      author={Zaslow, Eric},
       title={The proper {L}andau-{G}inzburg potential is the open mirror map},
        date={2024},
        ISSN={0001-8708,1090-2082},
     journal={Adv. Math.},
      volume={447},
       pages={Paper No. 109639, 69},
         url={https://doi.org/10.1016/j.aim.2024.109639},
      review={\MR{4739248}},
}

\bib{GS13}{article}{
      author={Gross, Mark},
      author={Siebert, Bernd},
       title={Logarithmic {G}romov-{W}itten invariants},
        date={2013},
        ISSN={0894-0347},
     journal={J. Amer. Math. Soc.},
      volume={26},
      number={2},
       pages={451\ndash 510},
         url={https://doi-org.ezproxy.uio.no/10.1090/S0894-0347-2012-00757-7},
      review={\MR{3011419}},
}

\bib{IP}{article}{
      author={{Ionel}, E.-N.},
      author={{Parker}, T.-H.},
       title={Relative {G}romov-{W}itten invariants},
        date={2003},
     journal={Ann. of Math.},
      volume={157},
      number={1},
       pages={45\ndash 96},
}

\bib{JTY}{article}{
      author={Jiang, Yunfeng},
      author={Tseng, Hsian-Hua},
      author={You, Fenglong},
       title={The quantum orbifold cohomology of toric stack bundles},
        date={2017},
        ISSN={0377-9017,1573-0530},
     journal={Lett. Math. Phys.},
      volume={107},
      number={3},
       pages={439\ndash 465},
         url={https://doi.org/10.1007/s11005-016-0903-1},
      review={\MR{3606511}},
}

\bib{Kim10}{incollection}{
      author={Kim, Bumsig},
       title={Logarithmic stable maps},
        date={2010},
   booktitle={New developments in algebraic geometry, integrable systems and
  mirror symmetry ({RIMS}, {K}yoto, 2008)},
      series={Adv. Stud. Pure Math.},
      volume={59},
   publisher={Math. Soc. Japan, Tokyo},
       pages={167\ndash 200},
         url={https://doi.org/10.2969/aspm/05910167},
      review={\MR{2683209}},
}

\bib{Koto}{article}{
      author={Koto, Yuki},
       title={Convergence and analytic decomposition of quantum cohomology of
  toric bundles},
        date={2024},
        ISSN={0001-8708,1090-2082},
     journal={Adv. Math.},
      volume={437},
       pages={Paper No. 109406, 66},
         url={https://doi.org/10.1016/j.aim.2023.109406},
      review={\MR{4671020}},
}

\bib{Li1}{article}{
      author={{Li}, J.},
       title={{Stable morphisms to singular schemes and relative stable
  morphisms}},
        date={2001},
     journal={J. Differential Geom.},
      volume={57},
       pages={509\ndash 578},
}

\bib{Li2}{article}{
      author={{Li}, J.},
       title={{A Degeneration formula of GW-invariants}},
        date={2002},
     journal={J. Differential Geom.},
      volume={60},
       pages={199\ndash 293},
}

\bib{LLW11}{article}{
      author={Lau, Siu-Cheong},
      author={Leung, Naichung~Conan},
      author={Wu, Baosen},
       title={A relation for {G}romov-{W}itten invariants of local
  {C}alabi-{Y}au threefolds},
        date={2011},
        ISSN={1073-2780},
     journal={Math. Res. Lett.},
      volume={18},
      number={5},
       pages={943\ndash 956},
         url={https://doi-org.ezproxy.uio.no/10.4310/MRL.2011.v18.n5.a12},
      review={\MR{2875867}},
}

\bib{LP}{article}{
      author={Lee, Y.-P.},
      author={Pandharipande, R.},
       title={A reconstruction theorem in quantum cohomology and quantum
  {$K$}-theory},
        date={2004},
        ISSN={0002-9327,1080-6377},
     journal={Amer. J. Math.},
      volume={126},
      number={6},
       pages={1367\ndash 1379},
  url={http://muse.jhu.edu/journals/american_journal_of_mathematics/v126/126.6lee.pdf},
      review={\MR{2102400}},
}

\bib{LR}{article}{
      author={{Li}, A.-M.},
      author={{Ruan}, Y.},
       title={Symplectic surgery and {G}romov-{W}itten invariants of
  {C}alabi-{Y}au 3-folds},
        date={2001},
     journal={Invent. Math.},
      volume={145},
      number={1},
       pages={151\ndash 218},
}

\bib{MP}{article}{
      author={Maulik, D.},
      author={Pandharipande, R.},
       title={A topological view of {G}romov-{W}itten theory},
        date={2006},
        ISSN={0040-9383},
     journal={Topology},
      volume={45},
      number={5},
       pages={887\ndash 918},
         url={https://doi-org.ezproxy.uio.no/10.1016/j.top.2006.06.002},
      review={\MR{2248516}},
}

\bib{NR}{article}{
      author={Nabijou, Navid},
      author={Ranganathan, Dhruv},
       title={Gromov–{W}itten theory with maximal contacts},
        date={2022},
     journal={Forum of Mathematics, Sigma},
      volume={10},
       pages={e5},
}

\bib{OP06}{article}{
      author={Okounkov, A.},
      author={Pandharipande, R.},
       title={Virasoro constraints for target curves},
        date={2006},
        ISSN={0020-9910,1432-1297},
     journal={Invent. Math.},
      volume={163},
      number={1},
       pages={47\ndash 108},
         url={https://doi.org/10.1007/s00222-005-0455-y},
      review={\MR{2208418}},
}

\bib{Takahashi}{article}{
      author={Takahashi, Nobuyoshi},
       title={Log mirror symmetry and local mirror symmetry},
        date={2001},
        ISSN={0010-3616,1432-0916},
     journal={Comm. Math. Phys.},
      volume={220},
      number={2},
       pages={293\ndash 299},
         url={https://doi.org/10.1007/PL00005567},
      review={\MR{1844627}},
}

\bib{Tseng}{article}{
      author={Tseng, Hsian-Hua},
       title={Orbifold quantum {R}iemann-{R}och, {L}efschetz and {S}erre},
        date={2010},
        ISSN={1465-3060},
     journal={Geom. Topol.},
      volume={14},
      number={1},
       pages={1\ndash 81},
         url={https://doi-org.proxy.lib.ohio-state.edu/10.2140/gt.2010.14.1},
}

\bib{TY18}{article}{
      author={Tseng, Hsian-Hua},
      author={You, Fenglong},
       title={Higher genus relative and orbifold {G}romov-{W}itten invariants},
        date={2020},
        ISSN={1465-3060},
     journal={Geom. Topol.},
      volume={24},
      number={6},
       pages={2749\ndash 2779},
         url={https://doi.org/10.2140/gt.2020.24.2749},
      review={\MR{4194303}},
}

\bib{TY20c}{article}{
      author={Tseng, Hsian-Hua},
      author={You, Fenglong},
       title={A {G}romov-{W}itten theory for simple normal-crossing pairs
  without log geometry},
        date={2023},
        ISSN={0010-3616,1432-0916},
     journal={Comm. Math. Phys.},
      volume={401},
      number={1},
       pages={803\ndash 839},
         url={https://doi.org/10.1007/s00220-023-04656-2},
      review={\MR{4604908}},
}

\bib{TY20b}{article}{
      author={Tseng, Hsian-Hua},
      author={You, Fenglong},
       title={A mirror theorem for multi-root stacks and applications},
        date={2023},
        ISSN={1022-1824,1420-9020},
     journal={Selecta Math. (N.S.)},
      volume={29},
      number={1},
       pages={Paper No. 6, 33},
         url={https://doi.org/10.1007/s00029-022-00809-8},
      review={\MR{4505176}},
}

\bib{vGGR}{article}{
      author={van Garrel, M.},
      author={Graber, T.},
      author={Ruddat, H.},
       title={{Local Gromov-Witten invariants are log invariants}},
        date={2019},
     journal={Adv. Math.},
      volume={350},
       pages={860\ndash 876},
}

\bib{Wang}{article}{
      author={Wang, Yu},
       title={Gross-{S}iebert intrinsic mirror ring for smooth log
  {C}alabi-{Y}au pairs},
        date={2022},
     journal={arXiv preprint arXiv:2209.15365},
}

\bib{You20}{article}{
      author={You, Fenglong},
       title={Relative {G}romov-{W}itten invariants and the enumerative meaning
  of mirror maps for toric {C}alabi-{Y}au orbifolds},
        date={2020},
        ISSN={0002-9947},
     journal={Trans. Amer. Math. Soc.},
      volume={373},
      number={11},
       pages={8259\ndash 8288},
         url={https://doi-org.ezproxy.uio.no/10.1090/tran/8196},
      review={\MR{4169688}},
}

\bib{You21}{article}{
      author={You, Fenglong},
       title={Gromov-{W}itten invariants of root stacks with mid-ages and the
  loop axiom},
        date={2021},
        ISSN={0001-8708,1090-2082},
     journal={Adv. Math.},
      volume={386},
       pages={Paper No. 107811, 25},
         url={https://doi.org/10.1016/j.aim.2021.107811},
      review={\MR{4267515}},
}

\bib{You24}{article}{
      author={You, Fenglong},
       title={Orbifold theta functions and mid-age invariants},
        date={2024},
     journal={arXiv preprint arXiv:2403.17077},
}

\bib{You22}{article}{
      author={You, Fenglong},
       title={The proper {L}andau--{G}inzburg potential, intrinsic mirror
  symmetry and the relative mirror map},
        date={2024},
     journal={Commun. Math. Phys. 405, 79.
  https://doi.org/10.1007/s00220-024-04954-3},
}

\end{biblist}
\end{bibdiv}

\end{document}